\documentclass[11pt,a4paper]{article}
\usepackage[OT1]{fontenc}
\usepackage[all]{xy}

\usepackage[english]{babel}
\usepackage{amsmath}
\usepackage{amsfonts}
\usepackage{amsthm}
\def\ZZ{\mathbb{Z}}

\def\d{ {\cal D} }
\def\c{ {\cal C} }
\def\a{ {\cal A} }
\def\h{ {\cal H} }
\def\ele{ {\cal L} }
\def\b{ {\cal B} }
\def\u{ {\cal U} }
\def\ii{ {\cal I} }
\def\t{ {\cal T} }

\def\e{ {\cal E} }
\def\p{ {\cal P} }

\def\k{ {\cal K} }
\def\f{ {\cal F} }
\def\x{ {\cal X} }

\def\v{ {\cal V} }
\def\w{ {\cal W} }

\def\q{ {\cal Q} }
\def\xx{ {\bf x} }

\def\XX{ {\bf X} }
\def\YY{ {\bf Y} }
\def\ZZ{ {\bf Z} }

\def\yy{ {\bf y} }
\def\zz{ {\bf z} }

\def\ww{ {\bf w} }

\def\ee{ {\bf e} }
\def\vv{ {\bf v} }

\def\aa{ {\bf a} }

\def\g1{ \mathfrak{g}_1  }

\newtheorem{teo}{Theorem}[section]
\newtheorem{prop}[teo]{Proposition}
\newtheorem{lem}[teo]{Lemma}

\newtheorem{defi}[teo]{Definition}
\theoremstyle{definition}
\newtheorem{rem}[teo]{Remark}

\newtheorem{ejems}[teo]{Examples}
\newtheorem{nota}[teo]{Notation}

\title{A non commutative K\"ahler structure on the Poincar\'e disk of a C$^*$-algebra}
\author{E. Andruchow, G. Corach and L. Recht}

\begin{document}

\maketitle 

\begin{abstract}
We study the Poincar\'e disk $\d=\{z\in\a: \|z\|<1\}$ of a C$^*$-algebra $\a$ as a homogeneous space under the action of an appropriate Banach-Lie group $\u(\theta)$ of $2\times 2$ matrices with entries in $\a$. We define on $\d$ a homogeneous K\"ahler structure in a non commutative sense. In particular, this K\"ahler structure defines on  $\d$  a homogeneous symplectic structure under the action of $\u(\theta)$. This action has a  moment map that we explicitly compute. In the presence of a trace in $\a$, we show that the moment map has a convex image when restricted to appropriate subgroups of $\u(\theta)$, resembling the classical result of Atiyah-Guillmien-Sternberg.
\end{abstract}
\bigskip

{\bf 2010 MSC:}  53D20, 53D50, 53C55, 53C60, 46L05, 58B20, 22E65, 46L08.

{\bf Keywords:}  Positive invertible operator, Poincar\'e disk, C$^*$-algebra, K\"ahler manifold, symplectic manifold. 
\section{Introduction}
 In the present paper we study the homogeneous geometry of the Poincar\'e disk  
$$
\d=\{z\in\a: \|z\|<1\}
$$
or, equivalently, the Poincar\'e half-space 
$$
\h=\{h\in\a: Im \ h \hbox{ is positive and invertible}\}
$$
 of a C$^*$-algebra  $\a$, under the action of an appropriate Banach-Lie group $\u(\theta)$ of $2\times 2$ matrices with entries in $\a$ (the group that preserves a quadratic form $\theta$ in $\a^2$ with values in $\a$). A main motivation for this is the following: if $G$ is the group of invertible elements of $\a$ and $G^+$ is the set of positive elements of $G$, then $G^+$ can be seen as the configuration space of a quantum mechanical system, whose elements are equivalent metrics on a fixed Hilbert space ${\cal L}$ - the element $a\in G^+$ defines the inner product $\langle \xi , \eta \rangle_a=\langle a\xi , \eta \rangle$, for $\xi,\eta\in{\cal L}$. The tangent bundle can be considered as  the phase space of the system, and also the bundle of observables associated to the different metrics in $\langle\ \cdot \ ,\ \cdot \ \rangle_a$. The set $\{(a,X): a\in G^+, X\in(TG^+)_a\}$ can be identified with 
$\h=\{X+i \ a: X\in\a, X^*=X, a\in G^+\}$. Thus, we have a clear identification between $TG^+$ and $\h$. On the other hand, $\d$ and $\h$ can be identified by means of a (Moebius) diffeomorphism.

The present paper is the third in a series. In the first paper \cite{tejemas} of this series, we studied some properties of $\d$ that derived from its structure  as a homogeneous reductive space: connection, geodesics, metric. In the second \cite{esfera de R} we concentrated on properties of $\d$ as an open subset of the projective line over the algebra $\a$. Mainly, the relationship between geodesics in $\d$ and the appropriate notion of cross-ratio of four points in the projective line over $\a$.

Here, we focus on the properties of $\d$ that derive from its structure as a {\it homogeneous K\"ahler manifold} in a non-commutative sense that we define. $\d$ turns out to be a homogeneous symplectic manifold under the action of the group $\u(\theta)$. This action  has a well defined moment map  that we compute explicitly. In the presence of a trace in $\a$, we show that this moment map  has a convex image when restricted to appropriate subgroups of $\u(\theta)$, resembling the classical result of Atiyah-Guillemin-Sternberg.

We briefly explain here the meaning of "non commutative K\"ahler structure".  These constructions are carried out in the space $\q_\rho$ - a space of idempotents in $M_2(\a)$ ($2\times 2$ matrices with entries in $\a$) which is equivalent to $\d$ and more convenient for these computations.  First we define a C$^*$-algebra bundle $\c$ on $\q_\rho$ whose fibers are C$^*$-algebras isomorphic to $\a$, that we call the {\it coefficient bundle}. It is the endomorphism bundle of the tautological bundle $\xi$ on $\q_\rho$:
\begin{equation}\label{tautologico}
\xi: (q,\xx)\mapsto q , \hbox{ for } a\in\q_\rho \hbox{ and } \xx\in R(q).
\end{equation}
Next, we define a Hilbertian inner product $\langle\ \ , \ \ \rangle$ on $\q_\rho$ such that $\langle X,Y\rangle_q\in\c_q$ (fiber of $\c$ over $q$) for $q\in\q_\rho$ and $X,Y\in(T\q_\rho)_q$. Then we show that the imaginary part $\omega$ of the Hilbertian product is a $\c$-valued $2$-form that is the curvature of the canonical connection of the tautological bundle $\xi$. In this context, the complex structure of the manifold  $\q_\rho$ comes from the embedding of $\d\simeq\q_\rho$ as an open set in the C$^*$-algebra $\a$.

If the C$^*$-algebra has a trace $\tau:\a\to\mathbb{C}$, the Hilbertian product produces a complex inner product $\tau\langle X, Y\rangle_q\in\mathbb{C}$ for $X,Y\in(T\q_\rho)_q$, $q\in\q_\rho$, and the $2$-form $\tau\omega_q(X,Y)$ turns out to be exact.

We also define a {\it non-commutative moment map} in the following sense. For $\tilde{a}$ in the Lie algebra $\mathfrak{U}(\theta)$ of the group $\u(\theta)$, and $q\in\q_\rho$, we define $f_{\tilde{a}}(q)\in\c_q$. This gives a map
$$
f:\q_\rho\times\mathfrak{U}(\theta)\to \c.
$$
We  call this map the moment map of $\q_\rho$. Again, in the presence of a trace $\tau$, the  map 
$$
\tau f:\q_\rho\times\mathfrak{U}(\theta)\to \mathbb{C}
 $$
defines a moment map $\q_\rho\to\mathfrak{U}(\theta)^*$ in the classical sense.

Section 2 contains a description of the set $\q_\rho$.  We also introduce the form $\theta:\a^2\times \a^2\to \a$ and its unitary group $\u(\theta)$, which turns out to be a  group of linear  fractional transformations, or the generalized symplectic group (see \cite{siegel}, \cite{krein}, \cite{potapov}, \cite{smuljan}, \cite{harpe}, \cite{young}).  We also introduce a  principal $\u_\a$-bundle ${\bf pr}:\k\to\q_\rho$, where   $\k$ is  contained in the {\it unit sphere} of $\theta$, and $\u_\a$ is the unitary group of $\a$. This principal bundle is a basic tool in the main constructions of this paper.

In Section 3 we describe three vector bundles related to $\q_\rho$ with their corresponding canonical connections. The so called  {\it coefficient bundle} $\c$ plays an essential role in this paper. In particular, we show that the tangent bundle $T\q_\rho$ is a $\c$-module. 

In Section 4 we introduce a complex structure in $\q_\rho$, which is compatible with the Hilbertian product in $\q_\rho$ defined in Section 5, where we introduce what we call the {\it symplectic form} $\omega$ of $\q_\rho$. 

The main result of Section 6 is a kind of pre-quantization of $\q_\rho$ \cite{woodhouse}: the curvature of the tautological bundle $\xi$  is, essentially, the symplectic form $\omega$. The moment map of $\q_\rho$ is studied in Section 7 and a $\u(\theta)$-invariant Finsler metric is studied in Section 8. 

The cases where $\a$ is $\mathbb{C}$ or a commutative $C^*$-algebra are considered in Section 9, where it is shown that these constructions coincide with those of the classical  Poincar\'e disk. Section 10 contains a kind of Atiyah-Guillemin-Sternberg's result, concerning the convexity of the restricted moment map in presence of a {\it valuation}, which is a map $\nu$ from the coefficient bundle $\c$ to a commutative C$^*$-algebra $\f$, with a tracial property. 

In Section 11  we do the following construction. First,  we identify the disk $\d$ with the Poincar\'e half-space $\h$ (see \cite{tejemas}) of $\a$, which is then also naturally  identified with $TG^+$. Then, in the presence of a trace in $\a$, we explicitly compute in $\h$ the symplectic $2$-form  $\omega$. Finally,  we construct an invariant Liouville $1$-form $\alpha$ on $TG^+$, and we show that $\omega=d\alpha$, as in the classical cotangent construction.

\section{The space $\q_\rho$}
We start with a review of polar decompositions of invertible elements  of a C$^*$-algebra. 
Every invertible element $g$ of a unital C$^*$-algebra factorizes in a unique way as $g=\lambda^2 \rho$, with $\lambda$ positive and $\rho$ unitary. In fact, $gg^*=\lambda^4$ and then $\rho=(gg^*)^{-1/2}g$. Analogously, $g=\rho\mu^2$ is the unique factorization with the (same) unitary on the left and the positive on the right: the formula $(gg^*)^{-1/2}=(g^*g)^{-1/2}g$ shows that the unitary part $\rho$ in both factorizations coincide. The positive parts $\lambda^2, \mu^2$ are related as $\lambda^2=(gg^*)^{1/2}=|g^*|$ and $\mu^2=(g^*g)^{1/2}=|g|$. In this paper we only need these remarks for the case when $g$ is a reflection, i.e., $g=g^{-1}$. In this case, it is straightforward to prove that $\rho=\rho^*=\rho^{-1}$ (a symmetry) and $\lambda\rho=\rho\lambda^{-1}$ (see \cite{cpr}, Section 3 for details). Therefore, if $g^{2}=1$ then $|g^*|=|g|^{-1}$ and $\rho^2=1$.

Notice also that there is a  natural correspondence between idempotents and reflections: $q^2=q$ if and only if $(2q-1)^2=1$. Thus, for every idempotent $q$ it holds that
$$
2q-1=\rho|2q-1|=|2q-1|^{-1}\rho=|2q-1|^{-1/2}\rho|2q-1|^{1/2}
$$
and $q=|2q-1|^{-1/2}p|2q-1|^{1/2}$, if $p=\frac12(1+\rho)$. 
\begin{nota}
$\lambda_q=|2q-1|^{-1/2}$, for every idempotent $q$.
\end{nota}
Thus, 
$2q-1=\lambda^2_q\rho=\rho\lambda_q^{-2}=\lambda_q\rho\lambda_q^{-1}$ and $q=\lambda_qp\lambda_q^{-1}$.

Let us apply these calculations in the context of the C$^*$-algebra $M_2(\a)$, with the fixed Hermitian reflection $\rho=\left(\begin{array}{cc} 1 & 0 \\ 0 & -1 \end{array}\right)$. Note that 
$p=\left(\begin{array}{cc} 1 & 0 \\ 0 & 0 \end{array}\right)$.

The sesquilinear $\a$-valued form $\theta:\a^2\times\a^2\to\a$ induced by $\rho$,
$$
\theta(\xx,\yy)=\langle \rho\xx,\yy\rangle=\xx^*\rho\yy=x_1^*y_1-x_2^*y_2
$$
is symmetric and non-degenerate because of the properties of $\rho$.

Every matrix $\tilde{a}\in M_2(\a)$ admits a unique adjoint with respect to the form $\theta$, namely $\tilde{a}^\sharp=\rho\tilde{a}^*\rho$. The group $\u(\theta)$ of invertible matrices which are unitary with respect to $\theta$ will play a central role in this paper:
$$
\u(\theta)=\{\tilde{a}\in Gl_2(\a): \tilde{a}^\sharp=\tilde{a}^{-1}\}.
$$
The group  $\u(\theta)$ has appeared frequently in the literature: see the papers by Siegel \cite{siegel}, Krein-Smuljan \cite{krein}, Potapov \cite{potapov}, Phillips \cite{phillips}, de la Harpe \cite{harpe}, Young \cite{young}, among many others. Its elements are also called generalized symplectic or linear fractional transformations. 

Among all the idempotent matrices $q\in M_2(\a)$, we shall consider those which satisfy
\begin{enumerate}
\item
$q^\sharp=q$, i.e. $\rho q=q\rho$.
\item
$\theta$ is positive semidefinite in the image $R(q)$ of $q$, and is negative semidefinite in the nullspace $N(q)$.
\end{enumerate}

The set of all these idempotents is denoted $\q_\rho$. We shall resume condition $2.$ above saying that $q$ {\it decomposes} $\theta$.

We characterize the elements of $\q_\rho$ as follows.
\begin{prop}
Let $q\in M_2(\a)$ be an idempotent matrix. The following conditions are equivalent:
\begin{enumerate}
\item
$q$ decomposes $\theta$.
\item
$\rho(2q-1)$ is positive.
\item
There exists a positive invertible matrix $\lambda\in M_2(\a)$, such that
$\lambda\rho=\rho\lambda^{-1}$ and $q=\lambda p\lambda^{-1}$.
\end{enumerate}
\end{prop}
For the proof, the reader is  referred to \cite{cpr}, Section 3. The positive matrix $\lambda$ of condition 3. is precisely the one appearing in the polar decomposition $2q-1=\lambda_q^2\rho$

\begin{rem}
Every idempotent $q\in M_2(\a)$ decomposes $\a^2$ as $\a^2=R(q)\oplus N(q)$. If $q^\sharp=q$, this decomposition is orthogonal with respect to $\theta$. If, additionally, $q\in\q_\rho$,    then 
$\theta$ is positive semidefinite in $R(q)$ and negative semidefinite in $N(q)$.
\end{rem}
 
Next, we shall characterize the positive part $\lambda_q$ of $2q-1$ for any $q\in\q_\rho$. Write $\lambda_q=\left(\begin{array}{cc} a & b^* \\ b & c \end{array}\right)$, with $a,b,c\in\a$; note that $a,c\ge 0$. If $q\in\q_\rho$, then  $\lambda_q\rho\lambda_q=\rho$, so
$$
\lambda_q \rho \lambda_q= \left(\begin{array}{cc} a^2-b^*b & -ab^*+b^*c \\ -ba+cb & -c^2+bb^* \end{array}\right) =\left(\begin{array}{cc} 1 & 0 \\ 0 & -1 \end{array}\right)
 $$
and we get  $a=(1+b^*b)^{1/2}$, $c=(1+bb^*)^{1/2}$.
Then
$$
\lambda_q=\left(\begin{array}{cc} (1+b^*b)^{1/2} & b^* \\ b & (1+bb^*)^{1/2} \end{array}\right) \ , \ \ \lambda_q^{-1}=\left(\begin{array}{cc} (1+b^*b)^{1/2} & -b^* \\ -b & (1+bb^*)^{1/2} \end{array}\right)
$$
and 
$$
q=\lambda_qp\lambda_q^{-1}=\left(\begin{array}{cc} 1+b^*b & -(1+b^*b)^{1/2}b^* \\ b(1+b^*b)^{1/2} & -bb^* \end{array}\right).
$$
Thus, the map
$$
b\mapsto \left(\begin{array}{cc} 1+b^*b & -(1+b^*b)^{1/2}b^* \\ b(1+b^*b)^{1/2} & -bb^* \end{array}\right)
$$
from $\a$ onto $\q_\rho$ is a parametrization, and the map $b\mapsto \left(\begin{array}{cc} (1+b^*b)^{1/2} & b^* \\ b & (1+bb^*)^{1/2} \end{array}\right)$ is a parametrizes the corresponding set $\{\lambda_q: q\in\q_\rho\}$.

Notice that for any $q\in\q_\rho$ it holds that $\lambda_q\in \u(\theta)\cap Gl_2(\a)^+$, where $Gl_2(\a)^+$ denotes the set of positive elements in $Gl_2(\a)$. Indeed, it is easy to see that $\u(\theta)\cap Gl_2(\a)^+=\{\lambda_q: q\in\q_\rho\}$.

\begin{rem}
Both columns of $q$ (in $\q_\rho$) are multiples of $\xx=\left(\begin{array}{c} (1+b^*b)^{1/2}\\b\end{array}\right)$, so $R(q)=[\xx]=\{\xx a: a\in\a\}$ and $\{\xx\}$ is a basis of $R(q)$, since $x_1=(1+b^*b)^{1/2}$ is invertible.
\end{rem}

The other important set in this study is the following subset of the  unit sphere of $\theta$
$$
\k=\k_\theta=\{\xx\in\a^2: \theta(\xx,\xx)=1, x_1 \hbox{ invertible}\}=\{\xx\in\a^2:\xx^*\rho\xx=1 , x_1 \hbox{ invertible}\}.
$$

$\k$ is a submanifold of $\a^2$, and for $\xx\in\k$ it holds that 
\begin{equation}\label{tangente esfera theta}
(T\k)_{\xx}=\{ \ZZ\in\a^2: \ZZ^*\rho \xx+\xx^* \rho\ZZ=0\}=\{\zz\in\a^2: Re\left(\theta(\xx,\zz)\right)=0\}.
\end{equation}

For each $\xx\in\k$, define
$$
p_\xx=\xx\xx^* \rho= \left(\begin{array}{cc} x_1x_1^* & -x_1x_2* \\ x_2x_1^* & -x_2x_2^* \end{array}\right)\in M_2(\a).
$$
Notice that $p_\xx(\zz)=\xx\ \theta(\xx,\zz)$ for all $\zz\in\a^2$.
\begin{prop}
For each $\xx\in\k$, $p_\xx\in\q_\rho$.
\end{prop}
\begin{proof}
Since $\xx^*\rho\xx=1$ for $\xx\in\k$, it is easy to check that $p_\xx$ is an idempotent. Also $p_\xx^\sharp=\rho p_\xx^*\rho=\rho \rho\xx\xx^*\rho=p_\xx$. It is clear that $\theta$ is positive semidefinite on $R(p_\xx)$: 
$$
\theta(p_\xx\yy,p_\xx\yy)=\yy^*\rho\xx\xx^*\rho\xx\xx^*\rho\yy=\yy^*\rho\xx\xx^*\rho\yy\ge 0.
$$
If $\zz\in N(p_\xx)$, then $\xx\xx^*\rho\zz=0$ and , since $x_1$ is invertible, $\xx^*\rho\zz=0$, i.e., $x_1^*z_1=x_2^*z_2$, or $z_1=(x_1^*)^{-1}x_2^*z_2$ Thus,
$$
\theta(\zz,\zz)=z_2^*x_2x_1^{-1}(x_1^*)^{-1}x_2^*z_2-z_2^*z_2=z_2^*(x_2x_1^{-1}(x_1^*)^{-1}x_2^*-1)z_2=z_2^*(aa^*-1)z_2,
$$
where $a=x_2x_1^{-1}$. It suffices to show that $aa^*< 1$, or equivalently, that $a^*a< 1$, i.e., $x_2^*x_2< x_1^*x_1$. This last inequality holds because $x_1^*x_1=1+x_2^*x_2$.
\end{proof}
\begin{rem}
If $\xx\in\k$ then $\{\xx\}$ is a basis for $R(p_\xx)$.
\end{rem}

Define the map ${\bf pr}:\k\to\q_\rho$, ${\bf pr}(\xx)=p_\xx$.

\begin{prop}
${\bf pr}(\k)=\q_\rho$.
\end{prop}
\begin{proof}
For $q\in\q_\rho$ it holds that $q=\left(\begin{array}{cc} 1+b^*b & -(1+b^*b)^{1/2}b^* \\ b(1+b^*b)^{1/2} & -bb^* \end{array}\right)$. Then
$$
\xx=\left(\begin{array}{c} (1+b^*b)^{1/2}\\ b \end{array}\right)\in \k
$$
satisfies $p_\xx=\xx\xx^*\rho=q$.
\end{proof}
\begin{prop}
Let $\xx,\yy\in\k$. Then $p_\xx=p_\yy$ if and only if $\yy=\xx u$ for some unitary $u\in\a$.
\end{prop}
\begin{proof}
If $\yy=\xx u$ for $u$ unitary, then $p_\yy=(\xx u)(\xx u)^*\rho=\xx uu^*\xx^*\rho=p_\xx$. 

For the converse statement, we claim that $u=\xx^*\rho\yy$ is a unitary element of $\a$ if $p_\xx=p_\yy$. Notice that in this case $p_\xx\yy=\yy$ and $p_\yy\xx=\xx$. Then
$$
uu^*=\xx^*\rho\yy\yy^*\rho\xx=\xx^*\rho p_\yy\xx=\xx^*\rho\xx=1,
$$
and
$$
u^*u=\yy^*\rho\xx\xx^*\rho\yy=\yy^*\rho p_\xx\yy=\yy^*\rho\yy=1.
$$
Thus $\yy=p_\xx\yy=\xx\xx^*\rho\yy=\xx u$.
\end{proof}

We need some facts concerning the tangent spaces of $\q_\rho$. Let us recall the differential structure of this space. If $\q$ denotes the set of idempotents of $M_2(\a)$ and $\p$ the set of all Hermitian idempotents, the polar decomposition defines a map $\Pi: \q \to \p$ as follows: if $q\in\q$ and $2q-1=\lambda_q^2\rho$ as before, then $\rho=\rho^{-1}=\rho^*$ and $p=\frac12(1+p)\in\p$; define $\Pi(q)=p$. Then
$$
\q_\rho=\Pi^{-1}(\{p\}).
$$
This fiber is a $C^\infty$-submanifold of $\q$, which is a closed and complemented submanifold of $M_2(\a)$; we refer the reader to  \cite{cpr} for a comprehensive study of these matters. Since $\q_\rho$ can also be described as
$$
\q_\rho=\{ q\in\q: \rho q^* \rho=q \ , \ (2q-1)\rho>0\}
$$
it follows that, for $q\in\q$ it holds
$$
(T\q_\rho)_q=\{\XX\in M_2(\a): \rho\XX^*\rho=\XX \ , \ \XX q+q\XX=\XX\}.
$$  
\begin{teo}
The map ${\bf pr}:\k\to\q_\rho$ is a principal fiber bundle with structure group $\u_\a$. It has a global cross section ${\bf sr}:\q_\rho\to\k$ defined by ${\bf sr}(q)=\lambda_q \ee_1$. The map $\Phi:\k\to\q_\rho\times\u_\a$, $\Phi(\xx)=(p_\xx,\theta(\xx,{\bf sr}(p_\xx)))$ is a trivialization for ${\bf pr}$, with inverse $\Phi^{-1}(q,u)={\bf sr}(q)u^*$.
\end{teo}
\begin{proof}
First note that the first column of $\lambda_q$, for $q\in\q_\rho$, belongs to $\k$:
$$
(\lambda_q \ee_1)^*\rho\lambda_q \ee_1=\ee_1^*\lambda_q\rho\lambda_q \ee_1=\ee_1^*\rho\ee_1=1,
$$
because $\lambda_q\rho\lambda_q=\rho$. Also the first coordinate of $\lambda_q\ee_1$ is $(1+b^*b)^{1/2}$, which is positive and invertible.

Next, we claim that ${\bf pr}(\lambda_1\ee_1)=q$, or equivalently, $p_{\lambda_q\ee_1}=q$. Indeed, since $p\rho=p$, we get
$$
{\bf pr}(\lambda_q\ee_1)=\lambda_q\ee_1\ee_1^*\lambda_q\rho=\lambda_q p\rho\lambda_q^{-1}=\lambda_q p\lambda_q^{-1}=q.
$$
Finally, note that $\theta(\xx,\lambda_{p_\xx}\ee_1)$ is a unitary element of $\a$, because 
$$
\xx^*\rho\lambda_{p_\xx}\ee_1\ee_1^*\lambda_{p_\xx}\rho\xx
=\xx^*\rho\lambda_{p_\xx}p\rho\lambda_{p_\xx}^{-1}\xx=\xx^*\rho\lambda_{p_\xx}p\lambda_{p_\xx}^{-1}\xx=\xx^*\rho p_\xx\xx=\xx^*\rho\xx=1,
$$
and similarly for the other product. The fact that $\Phi$ is a trivialization for ${\bf pr}$ with the given inverse is a straightforward verification.
\end{proof}
\begin{rem}
If $\xx\in\k$, then the square root  $\lambda_{p_\xx}$ of the positive part  of $2p_\xx-1$ is
$$
\lambda_{p_\xx}=\left(\begin{array}{cc} (1+x_1x_2^*x_2x_1^*)^{1/2} & x_1x_2^* \\
x_2x_1^* & (1+x_2x_1^*x_1x_2^*)^{1/2} \end{array}\right);
$$
and one has $\lambda_{p_\xx}p\lambda_{p_\xx}^{-1}=p_\xx$.
\end{rem}
We remarked before that the range of $q=p_\xx\in\q_\rho$ has basis $\{\xx\}$ ($\xx\in\k$), as right $\a$-submodule of $\a^2$. We complete this observation by showing that also $N(q)$ has a basis, and that together with $\xx$ they form a $\theta$-orthonormal basis of $\a^2$.
\begin{prop}
Let $\xx\in\k$. There exist a unique elements $\yy\in R(p_\xx)$ and $\zz\in N(p_\xx)$,
$\yy,\zz\in\k$, with $y_1,z_1$ positive and invertible.
\end{prop}
\begin{proof}
Straightforward computations show that
$$
\yy=\xx x_1^*(x_1x_1^*)^{-1/2} \ \hbox{ and } \zz=\left(\begin{array}{c} x_1x_1^*(1+x_2x_2^*)^{-1/2} \\ (1+x_2x_2^*)^{-1/2} \end{array} \right)
$$
satisfy the properties above. Let us prove that they are unique. If $\yy'\in R(p_\xx)\cap \k$ and $y_1$ is positive and invertible, then $\yy'=\yy a$ for some $a\in\a$, because $\yy$ is a basis for $[\xx]$. Note that  $1=\theta(\yy',\yy')=a^*\theta(\yy,\yy)a=a^*a$, and that $y'_1=y_1a$. Then $a$ is invertible, and therefore unitary. Then the identity $y_1'=y_1 a$ can be regarded as a polar decomposition of $y_1'$, which is already positive. By uniqueness of the polar decomposition, it follows that $a=1$.
The uniqueness of $\zz$ follows analogously.
\end{proof}

\subsection{A lifting form}\label{A lifting form}

Using the principal bundle ${\bf pr}:\k\to \q_\rho$  we introduce a {\it lifting form} $\kappa_{\xx}$, for $\xx\in\k$ and $q=p_\xx$. Given $\XX\in(T\q_\rho)_q$, $q=p_\xx\in\q_\rho$, we define 
\begin{equation}\label{levantado}
\kappa_{\xx}(\XX)=\XX\xx.
\end{equation}
This lifting $\kappa_{\xx}(\XX)$  is an element of $\a^2$, which belongs to the nullspace of $q$.  It holds that  $N(p_\xx)\subset (T\k)_\xx$. In fact, if $\zz\in N(q)$,  then 
$0=q(\zz)\xx\theta(\xx,\zz)$, and thus $\theta(\xx,\zz)=0$ because $x_1$ is invertible, i.e. $\zz\in(T\k)_\xx$.    Therefore $\kappa_\xx$ is a map $(T\q_\rho)_q\to N(q)$.

We have the following:
\begin{lem}
Let $\XX\in(T\q_\rho)_q$ and  $\xx\in\k$  with $q=p_\xx$. Then $\kappa_{\xx}:(T\q_\rho)_q\to N(q)$ satisfies
$$
(d {\bf pr})_\xx \kappa_{\xx}(\XX)=\XX.
$$
\end{lem}
\begin{proof}
Since ${\bf pr}(\xx)=\xx\xx^*\rho$, its differential   is given by 
$$
(d{\bf pr})_\xx(\YY)=\YY\xx^*\rho+ \xx\YY^*\rho. 
$$  
Then
$$
(d{\bf pr})_{\xx}\kappa_{\xx}(\XX)=\XX\xx\xx^*\rho+\xx\xx^*\XX^*\rho=\XX q+q\XX=\XX,
$$
because $\XX^*\rho=\rho\XX$ and $q=p_\xx$. 
\end{proof}

If we lift $\XX$ at another point $\yy=\xx u$ in the fiber of $q$, we get
$$
\kappa_{\yy}(\XX)=\XX\yy=\XX\xx u.
$$
In view of this, we see that the tangent vectors $\XX$ are given by pairs $(\xx,\zz)$, with $\zz\in N(q)$, subject to the identification
$$
(\xx,\zz)\sim(\xx u,\zz u).
$$
\subsection{The equivalence between $\q_\rho$, $\d$ and $\h$}\label{subseccion de mapas}

In this short subsection, we briefly recall from \cite{tejemas} the maps identifying the Poincar\'e disk and halfspace with the space $\q_\rho$, as well as the roles of $\u(\theta)$ and $\k$ in this identification. Let us start with $\d=\{z\in\a: \|z\|<1\}$. The counterpart of the fibration ${\bf pr}:\k\to\q_\rho$ is the map
$$
\hat{\pi}:\k\to\d , \ \hat{\pi}(\xx)=x_2x_1^{-1}.
$$
For $\xx\in\k$, the bijection $p_\xx\longleftrightarrow x_2x_1^{-1}$ is a (well defined) diffeomorphism between $\q_\rho$ and $\d$. Thus one has the commutative diagram
\begin{equation}\label{diagrama general}
\xymatrix{
& \k \ar[ld]_{\hat{\pi }}\ar[rd]^{{\bf pr}} \\
\d \ar[rr]^{\simeq}
&& \q_\rho ,
}
\end{equation}
The group $\u(\theta)$ acts on the three spaces, and the maps above are equivariant with respect to these actions.

With respect to the half-space $\h=\{h\in\a: Im\ h\in G^+\}$, the pertinent quadratic form to study this space is $\theta_H(\xx,\yy)=x_1^*y_2-x_2^*y_1$, which induces the group
$$
\u(\theta_H)=\{\tilde{a}\in Gl_2(\a): \theta_H(\tilde{a}\xx,\tilde{a}\yy)=\theta_H(\xx,\yy) , \hbox{ for all } \xx,\yy\in\a^2\} 
$$
and the sphere $\k_H=\{\xx\in\a^2: x_1\in G \hbox{ and } \theta(\xx,\xx)=1\}$. The fibration corresponding to ${\bf pr}$ in this context looks formally equal: $\k_H\ni\xx\mapsto x_2x_1^{-1}\in\h$. There is an analogous diagram as (\ref{diagrama general}), where the horizontal equivariant diffeomorphism is $x_2x_1^{-1}\longleftrightarrow p_\xx$.

Both settings are related by the unitary matrix
$$
U=\frac{1}{\sqrt2}\left(\begin{array}{cc} 1 & 1 \\ i & -i \end{array} \right).
$$ 
For instance, $\xx\in\k$ if and only if $U\xx\in\k_H$. 

Finally, the diffeomorphism between $\d$ and $\h$, induced by the identifications of both spaces with $\q_\rho$, $\q_{\rho'}$, respectively, is given by the Moebius transformation
$$
\Gamma:\h\to\d , \ \Gamma(h)=(1+ih)(1-ih)^{-1}
$$
with inverse 
$$
\Gamma^{-1}:\d\to\h , \ \Gamma^{-1}(z)=i(1-z)(1+z)^{-1}.
$$

\section{Three vector bundles constructed from ${\bf pr}$}

The reader is referred to the classical text of Kobayashi and Nomizu \cite{kobayashi} for all geometrical notions appearing in this section.

Let us describe the following setting, which will allow us to present a unified approach of the three main vector bundles which we shall use to study the geometry of $\q_\rho$. They are based on the principal bundle ${\bf pr}: \k\to\q_\rho$.

Let $\eta:F\to \q_\rho$ be a fibre bundle, consider the commutative diagram
\begin{equation}\label{esquema}
\begin{array}{ccc} F_{\k} & \stackrel{\tilde{{\bf pr}}}{\longrightarrow} & F \\ \eta^* \downarrow & \ & \downarrow \eta \\ \k & \stackrel{{\bf pr}}{\longrightarrow} & \q_\rho \end{array} ,
\end{equation}
where $F_\k=\{(f,\xx): f\in F , \xx\in\k, \eta(f)={\bf pr}(\xx)\}$. This diagram describes $\eta^*$ as the bundle over $\k$ induced  by ${\bf pr}$. Note that $F_{\k}$ carries a natural action from $\u_\a$, which lifts the action of $\u_\a$ over $\k$, and which presents the bundle $\eta:F\to\q_\rho$ as the space of orbits of $\eta^*:F_{\k}\to \k$ under this action.
 
We can argue that an element $f\in F$ over $q\in\q_\rho$ is represented by an element $\tilde{f}\in F_{\k}$ in the basis $\xx\in\k$ (${\bf pr}(\xx)=q$, $f$ lies over the same $\xx$).  More precisely,   $(\tilde{f},\xx)$ and $(\tilde{f}',\xx')$  represent the same element if and only if there exists $u\in\u_\a$ such that $\tilde{f}'=\tilde{f} u$ and $\xx'=\xx u$.

\begin{defi}
Given a bundle $\eta:F\to\q_\rho$, we call a $\k$-{\it presentation} of $\eta$ a bundle $\zeta:Z\to\k$, with an action of $\u_\a$ over $T$ which lifts the action of $\u_\a$ over $\k$, and a  bundle isomorphism $\Phi: Z\to F_{\k}$ preserving the action of $\u_\a$,
$$
\begin{array}{ccc} Z & \stackrel{\Phi}{\longrightarrow} & F_{\k}
\\ \zeta\searrow & \ & \swarrow \eta \\ \ & \k & \ \end{array} .
$$
\end{defi}

Let us give three examples of $\k$-presentations which will be relevant in this paper.

\subsection{ The  {\it tautological  bundle} over $\q_\rho$}
Consider
$$
\xi: \e\to\q_\rho,
$$
where $\e=\{(q,\xx)\in M_2(\a)\times \a^2: q\in\q_\rho , \xx\in R(q)\}$ and $\xi(q,\xx)=q$. Over each idempotent $q$ we put the vectors in the range of $q$. 

\begin{prop}\label{fibrado xi}
$\xi$ is a locally trivial vector bundle.
\end{prop}
\begin{proof}
Fix $q_0\in\q_\rho$. The map $\pi_{q_0}:\u(\theta)\to \q_\rho$, $\pi_{q_0}(\tilde{a})=\tilde{a}q_0\tilde{a}^{-1}$ has smooth local cross sections: there exists $r_{q_0}>0$ and a map 
$$
g_{q_0}:\b_{q_0}=\{q\in\q_\rho: \|q-q_0\|<r_{q_0}\}\to \u(\theta),
$$
 such that $g(q)q_0g^{-1}(q)=q$ (see \cite{prcurro}). Consider the  set $\v_{q_0}=\{(q, \xx)\in\xi: q\in\b_{q_0}\}$. Clearly $\v_{q_0}$ is open in $\xi$. The map
$$
\Phi_{q_0}:\v_{q_0}\to \b_{q_0}\times R(q_0)\  , \ \  \Phi_{q_0}(q,\xx)=(q,g^{-1}(q)\xx)
$$
is a local trivialization for $\xi$ near $q_0$. 
\end{proof}

Then we have the induced bundle over $\k$, and the diagram (\ref{esquema}) in this case:
$$
\begin{array}{ccc} \e_{\k} & \stackrel{\tilde{{\bf pr}}}{\longrightarrow} & \e \\ \xi^* \downarrow & \ & \downarrow \xi \\ \k & \stackrel{{\bf pr}}{\longrightarrow} & \q_\rho \end{array} ,
$$
Consider the product bundle $\k\times \a \to \a$ and the isomorphism $\beta:\k\times\a\to \e$, $\beta(\xx,a)=(p_\xx, \xx a)$. The action of $\u_\a$ on $\k\times\a$  is given by $(\xx,a)\cdot u=(\xx u, u^* a)$. The isomorphism $\beta$ gives a $\k$-presentation for $\e$. Thus,  an element $(q,\yy)$ of $\e$ is represented by a pair $(\xx,a)$, and two pairs $(\xx,a)$, $(\xx',a')$ represent the same element of $\e$ if and only if there exists $u\in\u_\a$ such that $\xx'=\xx u$ and $a'=u^* a$.

\subsection{The {\it  coefficient} bundle}

Analogously, we shall describe the bundle over $\q_\rho$, whose fiber over $q$ is the space of  right $\a$-module endomorphisms of $R(q)$, which we shall call the {\it coefficient bundle} 
$$
\gamma:\c\to\q_\rho .
$$
Formally, $\c=\{(q,\varphi)\in\q_\rho\times \ele_\a(R(q)): q\in\q_\rho\}$, where $\ele_\a(R(q))$ denotes the space of right $\a$-module endomorphisms of the module $R(q)$, and the map $\gamma:\c\to\q_\rho$ is $\gamma(q,\varphi)= q$. Given $\xx\in\k$ with $p_\xx=q$, an endomorphism $\varphi$ of $R(q)$  gives $\varphi(\xx)=\xx a$. We say that $\xx$ is a {\it basis} for $R(q)$, and that $a\in\a$ is the {\it  matrix} of $\varphi$ in the basis $\xx$.  If we change $\xx$ for $\yy=\xx u\in R(q)$ for $u\in\u_\a$, and denote by $b$ the matrix of $\varphi$ at $\yy$, $\varphi(\yy)=\yy b$, then 
\begin{equation}\label{x u}
\xx ub=\yy b=\varphi(\yy)=\varphi(\xx)u=\xx a u,
\end{equation}
i.e., $b=u^*au$. In other words, the endomorphisms can be identified with the pairs $(\xx,a)\in\k\times \a$, subject to the equivalence relation
\begin{equation}\label{relacion x u}
(\xx,a)\sim(\xx u, u^*au) \ , \ \ u \in\u_\a.
\end{equation}

\begin{nota}
We denote by 
$\Big[(\xx,a)\Big]$  the endomorphism of the module $R(p_\xx)$ which  has matrix $a$ in the basis $\xx$ .
\end{nota}

\begin{prop} The space $\c$ is a C$^\infty$ differentiable manifold and the map $\gamma:\c\to\q_\rho$ is a C$^\infty$  trivial fiber bundle. 
\end{prop}
\begin{proof}
We define a  smooth structure on $\c$ by  constructing  local charts. Fix $(q_0,\varphi_0)\in\c$. Consider the set
$$
\b_{q_0}=\{(q,\varphi)\in\c: \|q-q_0\|<r_{q_0}\},
$$ 
where $r_{q_0}$ is the radius given in the proof of Proposition \ref{fibrado xi}. There exists a map $g_{q_0}$ defined on such idempotents $q$, with values in $\u(\theta)$, such that  $g_{q_0}(q)q_0g_{q_0}^{-1}(q)=q$.  Note that this implies that the $\a$-module homomorphism $g_{q_0}$ (acting in $\a^2$) maps $R(q_0)$ onto $R(q)$. Then $g_{q_0}\varphi g_{q_0}^{-1}$ is a module homomorphism acting in $R(q_0)$. The  map
$$
\b_{q_0}\to \b_0=\{(q,\psi): q\in\q_\rho, \|q-q_0\|<r_{q_0}, \psi\in\ele_\a(R(q_0))\} \ , \  (q,\varphi)\mapsto (q,g_{q_0}\varphi g_{q_0}^{-1})
$$
is one to one.
The set $\b_0$ is clearly a C$^\infty$ manifold, because it is an open subset of $\q_\rho\times \ele_\a(R(q_0))$, which is itself a product of a manifold and a Banach space. If $\b_{q_0}\cap\b_{q_1}\ne\emptyset$, the transition map between $\b_0$ and $\b_1$, which is given by
$$
\b_0\ni (q,\psi)\mapsto (q,g_{q_1}g_{q_0}^{-1} \psi g_{q_0}g_{q_1}^{-1})\in\b_1,
$$ 
is clearly C$^\infty$. This defines a C$^\infty$ structure on $\c$, which makes the projection map $\gamma:\c\to\q_\rho$ a C$^\infty$ map.

Let us prove that it is a trivial bundle.
Consider the map $\Phi:\c\to\q_\rho\times\a$, $\Phi(q,\varphi)=(q, a)$, where $a\in\a$ is the matrix of $\varphi$ in the basis $\ww={\bf sr}(q)$ of $q$. This gives a global trivialization for $\gamma$.
\end{proof}

As in the previous cases, the coefficient bundle $\gamma:\c\to\q_\rho$ induces the bundle $\gamma^*:\c_{\k}\to\k$ and the commutative diagram (\ref{esquema}) in this setting becomes
$$
\begin{array}{ccc} \c_{\k} & \stackrel{\tilde{{\bf pr}}}{\longrightarrow} & \c \\ \gamma^* \downarrow & \ & \downarrow \gamma \\ \k & \stackrel{{\bf pr}}{\longrightarrow} & \q_\rho \end{array} .
$$
The fiber $(\c_{\k})_\xx$ over $\xx\in\k$ consists of the endomorphisms of $R(p_\xx)$. Therefore, one has the representation $\Phi:\k\times \a\to \c_{\k}$, given by
$$
\Phi(\xx, a)=(\xx, [(\xx,a)]).
$$
\subsection{The connection in the coefficient bundle.}

Since  $\c_q$ consists    of  endomorphisms $\xi_q$, $q\in\q_\rho$, the standard connection on $\c$ is given by the {\it Leibnitz formula}. Let $\XX\in T\q_\rho$    and $\xx=\xx(t)$  adjusted to $\XX$. Let $\sigma$ and $\varphi$ be cross sections of $\xi$ and $\c$, respectively.  We define the covariant derivative  $D_\XX \varphi$ by the rule
\begin{equation}\label{conexion fibrado coeficientes}
(D_\XX\varphi)\sigma:= D_\XX(\varphi\sigma)-\varphi(D_\XX\sigma).
\end{equation}

Alternatively, write $\sigma(t)=\xx(t)a(t)$ and $\varphi_t(\zz)=\varphi_t(\xx b)=\xx(t)\lambda_t\ b$
for $\lambda, b\in\a$. Let us compute both terms in (\ref{conexion fibrado coeficientes}).
$$
D_{_\XX} (\varphi\sigma)=D_{_\XX}(\varphi(\xx a))=D_{_\XX}(\xx \lambda a)=\xx\left(\dot{\lambda}a+\lambda\dot{a}+ \theta(\xx,\dot{\xx})\lambda a\right)
$$
On the other hand
$$
\varphi(D_{_\XX}\sigma)=\varphi(\xx \dot{a}+p_\xx(\dot{\xx})a)=\xx\lambda\dot{a}+\varphi(p_\xx(\dot{\xx})a.
$$
Put $p_\xx(\dot{\xx})=\xx b$, where $b=\theta(\xx,\dot{\xx})$, and then
$$
\varphi(D_{_\XX}\sigma)=\xx\lambda \dot{a}+\xx\lambda\theta(\xx,\dot{\xx})a.
$$
Substracting as in (\ref{conexion fibrado coeficientes}), one gets
\begin{equation}\label{conexion fibrado coeficientes bis}
(D_{_\XX} \varphi)\sigma=\xx\left(\dot{\lambda}+[\theta(\xx,\dot{\xx}),\lambda]\right) a.
\end{equation}

\subsection{The tangent bundle $T\q_\rho$}

In the fibration ${\bf pr}:\k\to\q_\rho$, note that for any $\xx\in\k$, the nullspace $N(p_\xx)$ is a subspace of $(T\k)_\xx$. 
We shall call $N(p_\xx)$ the {\it horizontal space} over $\xx$, ${\bf H}_\xx:=N(p_\xx)$. The differential
$$
(d{\bf pr})_\xx:(T\k)_\xx\to (T\q_\rho)_{p_\xx}
$$
is bijective when restricted to ${\bf H}_\xx$. Then we have the bundle $\nu:{\bf H}\to \k$, where ${\bf H}=\{(\zz,\xx): \xx\in\k, \ \zz\in {\bf H}_\xx\}$, and the map is given by $\nu(\zz,\xx)=\xx$. Thus we have the $\k$-presentation
$$
\begin{array}{ccc} {\bf H} & \stackrel{\Phi}{\longrightarrow} & (T\q_\rho)_{\k}
\\ \nu\searrow & \ & \swarrow \tau^* \\ \ & \k & \ \end{array} ,
$$
where $\tau: T\q_\rho\to \q_\rho$ is the tangent bundle and $\tau^*:(T\q_\rho)_{\k}\to\k$ is the induced bundle. The isomorphism $\Phi$ is given by the differential of ${\bf pr}$ restricted to ${\bf H}$, namely
\begin{equation}\label{ref 34}
\Phi(\xx,\zz)=(\xx, (d{\bf pr})_\xx(\zz)) \ , \hbox{ for } \xx\in\k \ , \ \zz\in {\bf H}_\xx .
\end{equation}
 Then a  tangent vector to $\q_\rho$ at $p_\xx$ is represented by a pair $(\xx,\zz)$, $\xx\in\k$, $\zz\in{\bf H}_\xx$;  $(\xx,\zz)$ and $(\xx u,\zz')$ represent the same vector if and only if $\zz'=\zz u$ ($u\in\u_\a$).

\subsubsection{$T\q_\rho$ as a $\c$-module}
Recall that the fiber $\c_q$ can be presented as the set of pairs $(\xx,a)\in\k\times\a$ modulo the equivalence $(\xx,a)\sim(\xx u, u^*au)$ for $u\in\u_\a$. Since each submodule $R(q)$ has a basis consisting of one element, then $\ele_\a(R(q))$ is in one to one correspondence with $\a$: having chosen a basis $\xx\in\k$ for $R(q)$, the map $\a\to\c_q$, $a\mapsto (\xx,a)$ is a $*$-isomorphism. Thus, each fiber $\c_q$ has the structure  of a C$^*$-algebra, namely $\a$. Clearly the structure does not depend on the choice of the basis; another basis $\xx u$  provides the same structure, because $Ad(u)$ is a $*$-automorphism of $\a$.

On the other hand, we saw before  that the tangent bundle has a similar presentation. Each $\XX\in(T\q_\rho)_q$ is given by the set of pairs $(\xx,\vv)\in\k\times \a^2$, modulo the equivalence $(\xx,\vv)\sim(\xx u,\vv u)$ for $u\in\u_\a$ and $\vv=\XX\xx=\kappa_\xx(\XX)$. Accordingly, we denote $\XX\in(T\q_\rho)_q$ as $[(\xx,\vv)]$.

\begin{defi}\label{def 35}
For $\XX\in(T\q_\rho)_q$ and $\varphi\in\ele_\a(R(q))$ we define the product
$$
\XX\cdot\varphi:=[(\xx,\vv a)]\in(T\q_\rho)_q
$$
if $\xx\in\k$, $p_\xx=q$, $\varphi=[(\xx,a)]$ and $\XX=[(\xx,\vv)]$.
\end{defi}
Observe that this operation is well defined: if we change the basis to $\xx u$, then $\varphi$ and $\XX$ are represented by $(\xx u, u^*au)$ and $(\xx u, \vv u)$ respectively, and the product in the new referential is $(\xx u,\vv u u^*au)=(\xx u,\vv au)$ which represents the same tangent vector $\XX\cdot \varphi$ as $(\xx,\vv a)$.

Thus, the tangent bundle $T\q_\rho$ can be regarded as a bundle of right modules  over the coefficient bundle, which is a bundle of C$^*$-algebras (isomorphic to $\a$).
\begin{rem} {\it The tautological bundle over $\k$}.

Let us consider the pullback $\e'={\bf pr}^*\e$  of $\e$ by ${\bf pr}$,  
$$
\xi ':\e '=\{(\xx, \vv)\in\k\times\a^2: p_\xx\vv=\vv\}\to \k \ , \ \ \xi'(\xx,\vv)=\xx.
$$
Pick $\xx\in\k$ and let $\e'_\xx$ be the fiber of  $\x'$ over $\xx$.  Then the map 
$$
\aa\to\e'_\xx\  , \ \  a\mapsto \xx a
$$
is an isomorphism of right $\a$-modules, with inverse $\vv\mapsto \theta(\xx,\vv)$. It can be regarded as a chart for $\e'_\xx$  associated to $\xx$.  Let $\yy\in\e_{p_\xx}$ (i.e., $p_\xx=p_\yy=q$); then there exists a unitary $u\in\u_\a$ such that $\yy=\xx u$.  The isomorphism between $\a$ and $\e'_\yy$  is $\a\ni b\mapsto \yy a$. The coordinates $a$ and $b$ of  a given $\vv\in\e_q$ are related by $\xx ub=\yy b=\vv=\xx a$, and thus, $a=ub$. Then we can regard the vectors $\vv\in\e_q$ as pairs $(\xx,a)\in\k\times\a$, with the following equivalence relation:
$$
(\xx,a)\sim (\xx u, u^*a) \ , \  \ u\in\u_\a.
$$
\end{rem}

\begin{rem} {\it The connection in the tautological bundle.}

The covariant derivative in the tautological bundle $\xi:\e\to \q_\rho$ is given by
$$
D_\XX\sigma=q(\XX\bullet\sigma),
$$
where $\sigma$ is a cross section for $\xi$ and $\XX\bullet \sigma$ is the directional derivative of $\sigma$, considering $\sigma$ as a function with values in $\a^2$. Let us write this explicitly, in referential terms: let $q(t)$  be a smooth curve in $\q_\rho$, with $\dot{q}=\XX$, and let $\xx(t)$ be a smooth lifting of $q(t)$ in $\k$,   i.e, $ p_{\xx(t)}=q(t)$. Then $\sigma(t)=\xx(t) a(t)$ where $a(t)\in\a$ is a smooth curve. We have
$$
D_{_\XX}\sigma=\frac{d}{dt} q(\xx(t) a(t))=q(\dot{\xx}(t)a(t)+\xx(t)\dot{a}(t))=\xx\left(\dot{a}(t)+\theta(\xx(t),\dot{\xx}(t))a(t)\right).
$$ 
This formula can be read as follows: the covariant derivative has two terms in the fibers of $\xi_q$, $\xx a$ and the component of $\dot{\xx}$ in $R(q)$ multiplied by $a$. Note also that since $q\xx=\xx$, $\dot{\xx}$ decomposes as $q\xx$ in $R(q)$ plus $\XX\xx$ in $N(q)$.
\end{rem}

\section{The complex  structure of $\q_\rho$}
In this section, we shall define a generalized complex structure on $\q_\rho$. This means a smooth map $q\mapsto \ii_q$, where
\begin{equation}\label{estructura compleja}
\ii_q:(T\q_\rho)_q\to(T\q_\rho)_q
\end{equation}
is a bounded linear map such that $\ii_q^2=-1_{(T\q_\rho)_q}$, with an integrability property.  For general ideas about complex structures in finite dimensional manifolds, we refer the reader to \cite{kobayashi}. 

Recall, from Section 2, the map 
$$
{\bf pr}:\k\to\q_\rho ,  \ {\bf pr}(\xx)=\xx\xx^*\rho
$$
and its tangent map $(d{\bf pr})_\xx\YY=(\YY\xx^*+\xx\YY^*)\rho$.

Fix $q\in\q_\rho$ and $\xx\in\k$ such that $q=p_\xx$. Then, for all $\ZZ\in(T\q_\rho)_q$,
 it holds that 
$$
(d{\bf pr})_\xx \ i\ZZ\xx=i\ (\ZZ\xx\xx^*- \xx\xx^*\ZZ^*)\rho=i\ (\ZZ q-q \ZZ)= i\ \ZZ (2q-1),
$$ 
because $q=p_\xx=\xx\xx^*\rho$, $\ZZ^*\rho=\rho\ZZ$ and $\ZZ q+q\ZZ=\ZZ$.
This shows that
$$
\ii_q\ZZ:=(d{\bf pr})_\xx \ i\ \ZZ_\xx=i\ \ZZ(2q-1)
$$
is a well defined map $\ii_q:(T\q_\rho)_q\to(T\q_\rho)_q$. The proof that $\ii_q(\ii_q\ZZ)=-\ZZ$ is a simple computation.
\begin{rem}
In the particular case when $q=p=\left(\begin{array}{cc} 1 & 0 \\ 0 & 0 \end{array}\right)$, and $\xx=\ee_1$, we have, in matrix form,
$\ZZ=\left( \begin{array}{cc} 0 & Z_1 \\ -Z_1^* & 0 \end{array}\right)$, and 
$$
\ii_{p}\ZZ=i\ZZ\rho=\left(\begin{array}{cc} 0 & -iZ_1 \\ -iZ_1^* & 0 \end{array} \right)= -i \left(\begin{array}{cc} 0 & Z_1 \\ Z_1^* & 0 \end{array} \right).
$$
\end{rem}
In order to see that this  complex structure is integrable we proceed as follows. 
First, we note that it is invariant under the action of the group $\u(\theta)$. 
Next, we identify $\q_\rho$ with the disk $\d=\{a\in\a: \|a\|<1\}\subset \a$ (see \cite{tejemas}, Section 7, in particular Theorem 7.3). The disk $\d$ has a natural complex structure, as an open subset of the Banach space $\a$, which is invariant for the action of $\u(\theta)$ (the identification $\q_\rho\simeq\d$ is $\u(\theta)$-equivariant \cite{tejemas}, Lemma 7.2).
In order to prove that our complex structure is {\it integrable}, it suffices to show that these structures coincide at the point  $q=p$, which corresponds to $0\in\d$ in the identification $\q_\rho\simeq\d$:
\begin{lem}
The complex multiplication defined in $\q_\rho$ corresponds, under the identification $\q_\rho\simeq\d$, to the usual multiplication by the imaginary constant $i$.
\end{lem}
\begin{proof}
As noted in the remark above, it suffices to prove this fact at $0\in\d$, which corresponds to $p\in\q_\rho$. Recall  from Subsection \ref{subseccion de mapas} the coordinate diagram (\ref{diagrama general}):
$$
\xymatrix{
& \k \ar[ld]_{\hat{\pi }}\ar[rd]^{{\bf pr}} \\
\d \ar[rr]^{\simeq}
&& \q_\rho ,
}
$$
The corresponding elements $0\in\d$  and $p\in\q_\rho$ lift to $\ee_1=\left(\begin{array}{c} 1 \\ 0 \end{array} \right)\in\k$. Let $\zeta\in(T\d)_0=\a$, and $\ZZ\in(T\q_\rho)_p$ corresponding to the same vertical  tangent vector $\vv=\left(\begin{array}{c} 0 \\ \zeta \end{array} \right)$ in $N(p)\subset(T\k)_{\ee_1}$. 
By the definition given above (\ref{estructura compleja}), multiplication by $i$ in $(T\q_\rho)_p$ consists in the usual multiplication $i \vv=\left(\begin{array}{c} 0 \\ i \ \zeta \end{array} \right) \in N(p)$. Let us verifiy that this induces the usual multiplication by $i$ in $\a$. The differential $(d\hat{\pi})_{\ee_1}$ of $\hat{\pi}: \k\to \d$, $\hat{\pi}(\xx)=x_2x_1^{-1}$ at $\ee_1$, for $\ww=\left(\begin{array}{c} w_1\\ w_2\end{array}\right)$, is given by
$$
(d\hat{\pi})_{\ee_1}(\ww)=w_2.
$$
Therefore $i \vv$ is mapped to $i \zeta$, the usual multiplication by $i$ in $\a$.
  \end{proof}

\section{The Hilbertian product in $\q_\rho$}

We define now on $\q_\rho$ a generalization of the notion of Hermitian structure. It consists of a Hilbertian product on each $(T\q_\rho)_q$ with values in $\c_q$, which is invariant under the action of $\u(\theta)$.

\begin{defi}\label{prod int}
Given $\XX,\YY\in(T\q_\rho)_q$, we put
$$
\langle \XX,\YY\rangle_q=\hbox{ the endomorphism of } R(q) \hbox{ given by the pair }(\xx,-\theta(\kappa_{\xx}(\XX),\kappa_{\xx}(\YY)))
$$
$$
=\Big[ (\xx,-\theta(\kappa_{\xx}(\XX),\kappa_{\xx}(\YY)))\Big],
$$
where $\xx\in\k$ is such that $p_\xx=q$, and the brackets denote the equivalence class as in (\ref{relacion x u}).
\end{defi}
Note that (regarding $\XX$ and $\YY$ as matrices in $M_2(\a)$)
$$
-\theta(\kappa_{\xx}(\XX),\kappa_{\xx}(\YY))=-(\XX\xx)^*\rho (\YY\xx)=-\xx^*\XX^*\rho\YY\xx.
$$
\begin{rem}
Some remarks are in order.
\begin{enumerate}
\item
The minus sign  is needed so that the form is positive. 
Indeed, recall that $\theta$ is negative in $N(q)$ and $\XX\xx\in N(q)$,
$$
\langle \XX,\XX\rangle_q=[(\xx, -\theta(\XX\xx,\XX\xx))].
$$
Recall that the C$^*$-algebra structure of the fibers $\c_q$ of $\c$ is that of $\a$: an endomorphism $\varphi$ of $\ele_\a(R(q)\simeq \a$ is positive if and only if any of its matrices is a positive element of $\a$. Recall from the definition of $\q_\rho$, that the projection $q$ decomposes $\theta$, i.e., $\theta$ is positive in $R(q)$ and negative in $N(q)$. On the other hand, as seen above, $\XX\xx\in N(q)$.
\item
If one changes $\xx$ for $\xx u$ for some $u\in\u_\a$, then 
$$
\theta(\kappa_{\xx u}(\XX u),\kappa_{\xx u}(\YY u))=u^* \theta(\kappa_{\xx}(\XX),\kappa_{\xx}(\YY))u,
$$
i.e. $\langle \XX,\YY\rangle_q$ is indeed an element of $\c_q$.
\item
This Hilbertian  product is linear with respect to the product defined in  \ref{def 35}: if $q=p_\xx\in\q_\rho$, $\XX, \YY\in(T\q_\rho)_q$ and $\varphi\in\ele_\a(R(q))$, then
$$
\langle\XX,\YY\cdot\varphi\rangle_q=\langle\XX,\YY\rangle_q\cdot\varphi.
$$
Indeed, if we use the basis $\xx\in\k$ for $R(q)$,  then $\YY\cdot\varphi$ is represented by the pair $(\xx,\YY \xx a)$, and 
$$
\theta(\kappa_{\xx}(\XX),\kappa_{\xx}(\YY\cdot\varphi))=(\XX\xx)^*\rho\YY\xx a,
$$
which is the  element of $\a$ that one uses to define $\langle\XX,\YY\rangle_q\cdot\varphi$ (in terms of the basis $\xx$).
\end{enumerate}
\end{rem}

\begin{teo}
The complex structure on $\q_\rho$  is compatible with the Hilbertian product in the following sense:
$$
\langle \ii_q\XX, \YY\rangle_q=-\langle \XX, \ii_q\YY\rangle_q,
$$
for $q=p_\xx\in\q_\rho$ and $\XX,\YY\in(T\q_\rho)_q$.
\end{teo}
\begin{proof}
We have that (for $\xx$ such that $p_\xx=q$)
$$
\langle\XX,\ii_q\YY\rangle_q=-\Big[(\xx,\theta(\XX\xx,\ii_q\YY\xx))\Big]=-\Big[(\xx,i\xx^*\XX^*\rho\YY(2q-1)\xx)\Big]=i\langle\XX,\YY\rangle_q.
$$
Here $(2q-1)\xx=\xx$. On the other hand
$$
\langle\ii_q\XX,\YY\rangle_q=-\Big[(\xx,\theta(\ii_q\XX\xx,\YY\xx))\Big]=-\Big[(\xx,i\XX(2q-1)\xx,\YY\xx)) \Big]=-i\langle\XX,\YY\rangle_q.
$$
\end{proof}
\begin{rem}
We saw earlier that the tangent space $(T\q_\rho)_q$ is a right module over the corresponding fiber $\c_q$ of the coefficient bundle. The scalar field $\mathbb{C}$ lies naturally inside $\c_q$ as 
$$
\mathbb{C}\ni z\longleftrightarrow \Big[(\xx,z I)\Big]\in\c_q,
$$
because $uz Iu^*=z I$ for all $u\in\u_\a$. Thus $(T\q_\rho)_q$ is a $\mathbb{C}$-vector space. It is clear that the complex structure in $T\q_\rho$ which we defined above (in terms of the identification $\Phi$ of (\ref{ref 34}))  coincides with the one induced by  operator $\ii$.
\end{rem}
\begin{rem}
One might wish extend the argument above, namely, to use the identification $\q_\rho\simeq\d\subset\a$ in order to endow the tangent bundle $T\q_\rho$ with a right action of the algebra $\a$. However, this action depends on the immersion (at the tangent level) and is not intrinsic. It works in the case of $\mathbb{C}$, as it  works  also for the center of $\a$: endomorphisms with matrix in the center of $\a$ have the same matrix for any basis of the given submodule.
\end{rem}
\subsection{The symplectic form $\omega$}
If $\XX,\YY\in(T\q_\rho)_q$, the decomposition in selfadjoint and anti-selfadjoint parts
$$
\langle\XX,\YY\rangle_q=(\XX,\YY)_q+ i \omega(\XX,\YY)_q
$$
provides a  generalized {\it Riemannian} form
$$
(\XX,\YY)_q=Re \langle\XX,\YY\rangle_q
$$
and a generalized symplectic form
$$
\omega(\XX,\YY)_q=Im \langle\XX,\YY\rangle_q .
$$
In this paper we are interested in $\omega$.

\section{The curvature of  the canonical connection in $\e$}\label{seccion curvatura}

To simplify the computation, we make the assumption that given $\XX, \YY\in (T\q_\rho)_q$, we can construct  a smooth map $q(t,s)\in\q_\rho$ such that
$$
q(0,0)=q , \ \frac{\partial}{\partial t}q|_{(0,0)} =\XX  \ \ \hbox{ and } \ \ \frac{\partial}{\partial s}q|_{(0,0)}=\YY,
$$
and this map $q(t,s)$ is lifted to a map $\xx(t,s)\in\k$.
We assume also the existence of a cross section $\sigma$ defined on a neighbourhood of $q$. $\sigma=\xx a$, where both $\xx$ and $a$ are functions of $(t,s)$. To abbreviate, we shall write with a dot $\dot{}$ the derivatives with respect to $t$, and with a tilde $'$ the derivatives with respect to $s$. Then, differentiating
$$
D_{_\XX}\sigma=\xx(\dot{a}+\theta(\xx,\dot{\xx}) a)
$$
with respect to $s$, we get
$$
\xx\left( \dot{a}'+\theta(\xx',\dot{\xx})a +\theta(\xx,\dot{\xx}'a+\theta(\xx,\dot{\xx})a'+\theta(\xx,\xx')(\dot{a}+\theta(\xx,\dot{\xx})a)\right)=D_{_\YY} D_{_\XX} \sigma.
$$
Interchanging $'$ with $\dot{}$ we get
$$ 
\xx\left( \dot{a'}+\theta(\dot{\xx},\xx')a +\theta(\xx,\dot{\xx'}a+\theta(\xx,\xx')\dot{a}+\theta(\xx,\dot{\xx})(a'+\theta(\xx,\xx')a)\right)=D_{_\XX} D_{_\YY} \sigma.
$$

This should be specialized at $(t,s)=(0,0)$. Clearly, the lifting of $q(t,s)$ can be done in order that $\xx'$ and $\dot{\xx}$ are {\it horizontal} at $(0,0)$ (i.e., that they belong to $N(q)$). Then we get
$$
R(\XX,\YY)\sigma=D_{_\XX} D_{_\YY} \sigma-D_{_\YY} D_{_\XX} \sigma=\xx(\theta(\dot{\xx},\xx')-\theta(\xx',\dot{\xx}))a.
$$
Since $\dot{\xx}$ and $\xx'$ are horizontal, we have that 
$$
\dot{\xx}=\kappa_{\xx}(\XX)=\XX\xx \  \hbox{ and } \ \xx'=\YY\xx.
$$
Then 
\begin{equation}\label{curvatura}
R(\XX,\YY)\sigma=\Big[(\xx,\xx(\theta(\XX\xx,\YY\xx)-\theta(\YY\xx,\XX\xx))a)\Big]=\Big[(\xx,\xx\theta(\xx,[\XX,\YY]\xx)a)\Big].
\end{equation}
Here we use that $\XX$ and $\YY$, being tangent vectors of $\q_\rho$, are $\theta$-symmetric, i.e., $\XX^*\rho=\rho\XX$. 

Note that $R(\XX,\YY)$ is an endomorphism of $\e_q$.  Recall the Hilbertian $\c$-valued product in $\q_\rho$
$$
\langle \XX,\YY\rangle_q =-\Big[ (\xx, \theta(\XX\xx,\YY\xx)\Big]
$$
for $\xx\in\k$ such that $p_\xx=q$, which is an element of $\c_q$, i.e., an endomorphism of $\e_q$. Its imaginary part is 
$$
Im   \langle \XX,\YY\rangle_q=-\frac{i}{2}\Big[ (\xx, -\theta(\XX\xx,\YY\xx)+\theta(\XX\xx,\YY\xx)^*)\Big]=-\frac{i}{2}\Big[ (\xx,\theta(\YY\xx,\XX\xx)-\theta(\XX\xx,\YY\xx))\Big]
$$ 
$$
=\frac{i}{2}\Big[(\xx,\theta(\xx,[\XX,\YY]\xx)\Big].
$$
Thus, we have proved the following  result, which is a kind of prequantization of $\q_\rho$ \cite{woodhouse}:
\begin{teo}\label{61}
The curvature of the tautological bundle $\e$ and the Hilbertian product in $\q_\rho$ are related by the following formula
$$
\frac{i}{2}\  R(\XX,\YY)_q=Im \langle\XX,\YY\rangle_q=\omega(\XX,\YY)_q.
$$
\end{teo}

\section{The moment map}
Let us recall the Banach-Lie algebra $\mathfrak{U}(\theta)$ of the Banach-Lie group $\u(\theta)$ defined in Section 1. Note that this Lie algebra decomposes
$$
\mathfrak{U}(\theta)=\mathfrak{U}_0(\theta)\oplus \mathfrak{U}_1(\theta),
$$
where 
\begin{equation}\label{mathfrak}
\mathfrak{U}_0(\theta)=\{\left(\begin{array}{cc} a_1 & 0 \\ 0 & a_2\end{array}\right)\in M_2(\a): a_1^*=-a_1 , a_2^*=-a_2\}
\end{equation}
and
$$
\mathfrak{U}_1(\theta)=\{\left(\begin{array}{cc} 0 & a^* \\ a & 0\end{array}\right)\in M_2(\a): a\in\a\}.
$$

If $\tilde{a}\in \mathfrak{U}(\theta)$ and $q\in\q_\rho$, we put
\begin{equation}\label{momento}
f_{\tilde{a}}(q)=\frac{1}{2i}\Big[(\xx, \theta(\xx,\tilde{a}\xx))\Big]=\frac{1}{2i}\Big[(\xx, \xx^*\rho \tilde{a} \xx)\Big],
\end{equation}
for $\xx\in\k$ such that $p_\xx=q$. Again, a simple computation shows that if one chooses $\xx u$ instead, the element $\theta(\xx,\tilde{a}\xx)$ varies accordingly: 
$$
\theta(\xx u ,\tilde{a}\xx u)= \langle \rho \xx u, \tilde{a}\xx u\rangle=u^*\langle \rho\xx, \tilde{a}\xx\rangle u=u^*\theta(\xx,\tilde{a}\xx)u.
$$
Thus, $f_{\tilde{a}}(q)\in\c_q$. We call $f_{\tilde{a}}:\q_\rho\to \c$ is {\it the moment map} of the generalized symplectic manifold $\q_\rho$.

\begin{prop}
The moment map is equivariant with respect to the action of $\u(\theta)$: if $\tilde{m}\in\u(\theta)$, $q\in\q_\rho$ and $\tilde{a}\in \mathfrak{U}(\theta)$,
$$
f_{\tilde{m}\tilde{a}\tilde{m}^{-1}}(\tilde{m}q\tilde{m}^{-1})=f_{\tilde{a}}(q).
$$
\end{prop}
\begin{proof}
Pick $\xx\in\k$ such that $p_\xx=q$. Recall that $\tilde{m}\in\u(\theta)$ means that $\rho \tilde{m}^{-1}=\tilde{m}^*\rho$. Note that for $\yy\in\a^2$,
$$
\tilde{m}q\tilde{m}^{-1}\yy=\tilde{m}(\xx \theta( \xx, \tilde{m}^{-1}\yy))=\tilde{m}\xx \theta(\tilde{m}\xx,\yy)=p_{\tilde{m}\xx}\yy,
$$
where clearly $\tilde{m}\xx\in\k$ (see \cite{tejemas}; in general $\tilde{a}\xx\in\k$ for any $\tilde{a}\in\u(\theta)$, $\xx\in\k$). Then
$$
f_{\tilde{m}\tilde{a}\tilde{m}^{-1}}(\tilde{m}q\tilde{m}^{-1})=\Big[(\tilde{m}\xx, \frac{1}{2i}\theta(\tilde{m}\xx, \tilde{m}\tilde{a}\tilde{m}^{-1}\tilde{m}\xx))\Big]=\Big[(\tilde{m}\xx, \frac{1}{2i}\theta(\xx, \tilde{a}\xx))\Big]=f_{\tilde{a}}(q).
$$
\end{proof}

Let us compute the covariant derivative of $f_{\tilde{a}}$. We use a horizontal lifting $\xx(t)\in\k$ of the curve $q(t)$ in the direction of $\XX$ at $q$ (i.e., $q(0)=q$, $\dot{q}(0)=\XX$ and $\dot{\xx}(t)\in N(q(t))$). Then
$$
D_\XX f_{\tilde{a}}=\Big[(\xx, \frac{1}{2i}\left(\theta(\dot{\xx}, \tilde{a}\xx)+\theta(\xx,\tilde{a}\dot{\xx}))\right)\Big]=
\Big[(\xx, \frac{1}{2i}\left(\theta(\dot{\xx}, \tilde{a}\xx)-\theta(\tilde{a}\xx,\dot{\xx}))\right)\Big]=\Big[(\xx, \frac{1}{2i}\left(\theta(\dot{\xx}, \tilde{a}\xx)-\theta(\dot{\xx},\tilde{a}\xx)^*)\right)\Big]
$$
$$
=\Big[(\xx,\  Im \  \theta(\dot{\xx}, \tilde{a}\xx))\Big],
$$
where we use that $\tilde{a}$ is $\theta$-anti-symmetric. Since $\dot{\xx}$ is horizontal, it holds that $\dot{\xx}=\XX\xx$. Then
$$
D_\XX f_{\tilde{a}}=\Big[(\xx, \ Im \ \theta(\XX\xx,\tilde{a}\xx)\Big].
$$

On the other hand, when computing the curvature of the canonical connection of $\e$, we proved that for any $\XX,\YY\in(T\q_\rho)_q$
$$
Im \ \langle \XX,\YY\rangle_q=\Big[(\xx,\frac{1}{2i}\left(\theta(\XX\xx,\YY\xx)-\theta(\XX\xx,\YY\xx)^*)\right)\Big]=-\Big[(\xx, Im \ \theta(\XX\xx,\YY\xx))\Big].
$$ 
Now we consider for each $\tilde{a}\in\mathfrak{U}(\theta)$ the vector field in $\q_\rho$
$$
\XX_{\tilde{a}}(q)=(d\pi_q)_1(\tilde{a})=[\tilde{a},q],
$$
where $\pi_q:\u(\theta)\to\q_\rho$ denotes the action map: $\pi_q(\tilde{m})=\tilde{m} q \tilde{m}^{-1}$. Applying the formula above with $\XX=\XX_{\tilde{a}}$, we get
$$
\omega(\XX_{\tilde{a}},\YY)_q=Im \ \langle \XX_{\tilde{a}},\YY\rangle_q=\Big[(\xx,- Im \ \theta(\XX_{\tilde{a}}\xx,\YY\xx))\Big].
$$
Observe that $\XX_{\tilde{a}}(q)\xx=[\tilde{a},q]\xx=\tilde{a}q\xx-q\tilde{a}\xx=\tilde{a}\xx-q\tilde{a}\xx$ (recall that $q\xx=p_\xx\xx=\xx$). Note also that $\theta(q\tilde{a}\xx,\YY\xx)=0$ because $ \YY\xx\in N(q)$. Thus $\theta(\XX_{\tilde{a}}\xx,\YY\xx)=\theta(\tilde{a}\xx,\YY\xx)$. Therefore
$$
\omega(\XX_{\tilde{a}},\YY)_q=\Big[(\xx,-Im \ \theta(\tilde{a}\xx,\YY\xx))\Big].
$$
We proved before that  $D_\YY f_{\tilde{a}}=\Big[\xx,  \ Im \ \theta(\YY\xx,\tilde{a}\xx))\Big]$. Thus, we have shown that 
$$
D_\YY f_{\tilde{a}}=  \omega(\XX_{\tilde{a}},\YY).
$$
Suppose now that $f$ is a section of $\c$ over some open subset $\w$ of $\q_\rho$. Consider the covariant derivative $Df$ as a $1$-form on $\w$ with values in $\c$: $Df(\XX)=D_\XX f$. We say that a field $\XX_f$ is the {\it symplectic gradient} of $f$ if 
$$
D_\XX f=\omega(\XX,\XX_f).
$$

We are mainly interested in the case of functions of the form $f_{\tilde{a}}$, for $\tilde{a}\in\mathfrak{U}(\theta)$,
$$
f_{\tilde{a}}(q)=\frac{1}{2 i}\Big[(\xx,\theta(\xx,\tilde{a}\xx))\Big]=\frac{1}{2 i}\Big[(\xx,\xx^*\rho\tilde{a}\xx)\Big],
$$
where $q\in\q_\rho$ and $\xx\in\k$ satisfies $p_\xx=q$. Recall that the Lie algebra $\mathfrak{U}(\theta)$ consists of all matrices $\tilde{a}=\left(\begin{array}{cc} a_{11} & \alpha \\ \alpha^*  & a_{22} \end{array}\right)$, where $a_{ii}^*=-a_{ii}$. Recall also the identity $D_\XX f_{\tilde{a}}= \omega(\XX,\XX_{\tilde{a}})$. This says that $\XX_{\tilde{a}}$ is the symplectic gradient of $f_{\tilde{a}}$.

Recall also that $\XX_{\tilde{a}}(q)=[\tilde{a},q]$. Now given $\tilde{a},\tilde{b}\in\mathfrak{U}(\theta)$, we want to find the value of $\omega(\XX_{\tilde{a}},\XX_{\tilde{b}})$. at $q=\left(\begin{array}{cc} 1 & 0 \\ 0 & 0 \end{array}\right)$, and correspondingly $\xx=\ee_1\in\k$. Since
$$
 X_{\tilde{a}}=[\tilde{a},q]=\left(\begin{array}{cc} 0 & -\alpha \\ \alpha^* & 0 \end{array}\right) \ , \ \ X_{\tilde{b}}=[\tilde{b},q]=\left(\begin{array}{cc} 0 & -\beta \\ \beta^* & 0 \end{array}\right)
$$
for some $\alpha,\beta\in\a$, then
$$
\langle\XX_{\tilde{a}},\XX_{\tilde{b}}\rangle_q=\theta\big( \left(\begin{array}{c} 0 \\ \alpha^*\end{array} \right), \left(\begin{array}{c} 0 \\ \beta^*\end{array} \right)\Big)=-\alpha\beta^*
$$
and 
$$
\omega(\XX_{\tilde{a}},\XX_{\tilde{b}})_q=Im \ \langle\XX_{\tilde{a}},\XX_{\tilde{b}}\rangle_q=\frac{1}{2i}(\beta\alpha^*-\alpha\beta^*).
$$
We want to establish the relationship between $\omega(\XX_{\tilde{a}},\XX_{\tilde{b}})$ and $f_{[\tilde{a},\tilde{b}]}$. To this effect, observe first that 
$$
\theta(\ee_1,[\tilde{a},\tilde{b}]\ee_1)=-\theta(\tilde{a}\ee_1,\tilde{b}\ee_1)+\theta(\tilde{b}\ee_1,\tilde{a}\ee_1)=-\theta\big(\left(\begin{array}{c} a_{11} \\ \alpha^* \end{array}\right), \left(\begin{array}{c} b_{11} \\ \beta^* \end{array}\right)\big)+ \theta\big(\left(\begin{array}{c} b_{11} \\ \beta^* \end{array}\right)\left(\begin{array}{c} a_{11} \\ \alpha^* \end{array}\right)\big)
$$
$$
\alpha\beta^*-\beta\alpha^*+ a_{11}b_{11}-b_{11}a_{11}.
$$
Then $\omega(\XX_{\tilde{a}},\XX_{\tilde{b}})=\frac{1}{2i}(\beta\alpha^*-\alpha\beta^*)$ and $f_{[\tilde{a},\tilde{b}]}=\frac{1}{2i}(\alpha\beta^*-\beta\alpha^*+a_{11}b_{11}-b_{11}a_{11})$. So
$$
\omega(\XX_{\tilde{a}},\XX_{\tilde{b}})=-f_{[\tilde{a},\tilde{b}]}+\frac{1}{2i}(a_{11}b_{11}-b_{11}a_{11}).
$$
Finally, $f_{\tilde{a}}(q)=\frac{1}{2i}\theta(\ee_1,\tilde{a}\ee_1)=\frac{1}{2i}a_{11}$, and
\begin{equation}\label{omega}
\omega(\XX_{\tilde{a}},\XX_{\tilde{b}})=-f_{[\tilde{a},\tilde{b}]}+2i [f_{\tilde{a}},f_{\tilde{b}}].
\end{equation}
This equality holds at $q=\left(\begin{array}{cc} 1 & 0 \\ 0 & 0 \end{array}\right)$, but a simple argument shows that (\ref{omega}) holds at every $q\in\q_\rho$. Indeed, each of  the terms involved is a function in $q\in\q_\rho$ with $\tilde{a},\tilde{b}\in\mathfrak{U}(\theta)$ fixed. If we change $q$ by $\tilde{m}q\tilde{m}^{-1}$ for $\tilde{m}\in\u(\theta)$, the values change by an inner automorphism. Finally, since $\u(\theta)$ acts transitively on $\q_\rho$, our claim is proven.

Given two sections $f,g$ with their respective symplectic gradients $\XX_f$ and $\XX_g$, the {\it Poisson bracket} $\{f,g\}$ is defined as
$$
\{f,g\}=\omega(\XX_f,\XX_g).
$$
Then,  we get 
$$
\{f_{\tilde{a}},f_{\tilde{b}}\}=-f_{[\tilde{a},\tilde{b}]}+2i [f_{\tilde{a}},f_{\tilde{b}}].
$$
The term $2i [f_{\tilde{a}},f_{\tilde{b}}]$ occurs because of the non-commutativity of the C$^*$-algebras $\c_q$, where the moment map takes its values.

Summarizing the facts of the last sections: 

\begin{teo}\label{teorema momentos}
Consider the manifold $\q_\rho$, with the tautological bundle $\e$ and the coefficient bundle $\c$. Then
\begin{enumerate}
\item
There exist invariant connections in $\e$ and $\c$, linked by Leibnitz' rule (see (\ref{conexion fibrado coeficientes})).
\item
There exists a Hilbertian product in $T\q_\rho$, with values in $\c$. This product is compatible with the right $\c$-module structure of $T\q_\rho$.
\item
The imaginary part $\omega$ of the Hilbertian product in $T\q_\rho$ is the curvature of the tautological connection of $\e$. In this sense, the symplectic form $\omega$ is {\it exact}.
\item
The map $f_{\tilde{a}}$ ($\tilde{a}\in\mathfrak{U}(\theta)$) is a  moment map: the field $\x_{\tilde{a}}$ is   {\it symplectic gradient}  of the function (cross section for $\c$)  $f_{\tilde{a}}$. Here  gradients are computed using the covariant derivative.
\end{enumerate}
\end{teo}

\section{The invariant Finsler structure}
The homogeneous space $\q_\rho$ studied in \cite{cpr} has an invariant Finsler structure, i.e., a  continuous $\u(\theta)$-invariant distribution of norms in the tangent spaces. If $q\in\q_\rho$ and $\XX\in(T\q_\rho)_q$, define the norm 
$$
\|X\|_q=\||2q-I|^{-1/2}\XX|2q-I|^{1/2}\|.
$$
We recall that it is precisely this structure, which when translated to the disk $\d$, gives the Poincar\'e metric of the disk \cite{tejemas}. In this subsection we shall see that the Hilbertian product just defined, induces also in a natural way the same Finsler structure. To prove this fact, we must first indicate how to compute the norm of a given endomorphism $\varphi\in\ele_\a(R(q))$.
\begin{defi}\label{norma de phi}
Let $\varphi\in\ele_\a(R(q))$. Then
$$
|\varphi|:=\sup_{0\ne\yy\in R(q)}\displaystyle{\frac{\|\theta(\varphi(\yy),\varphi(\yy))\|^{1/2}}{\|\theta(\yy,\yy)\|^{1/2}}}.
$$
\end{defi}
Alternatively, $\theta$ is a C$^*$-Hilbert module (positively) inner product in $R(q)$. Thus, the above formula is just the usual way to compute the norm of an endomorphism, when the module $R(q)$ is endowed with the C$^*$-module norm.

We show that, in the presence of a basis $\xx\in\k$ of $R(q)$ and a matrix $a\in\a$ for $\varphi$, the norm of the endomorphism is the norm of the matrix.
\begin{lem}
If $\varphi=[(\xx,a)]$, then $|\varphi|=\|a\|$.
\end{lem}
\begin{proof}
Let $\yy=\xx b\ne 0$ in $R(q)$. Then $\varphi(\yy)=\xx ab$. Thus
$$
\displaystyle{\frac{\|\theta(\varphi(\yy),\varphi(\yy))\|^{1/2}}{\|\theta(\yy,\yy)\|^{1/2}}}=\displaystyle{\frac{\|\theta(\xx ab,\xx ab)\|^{1/2}}{\|\theta(\xx b,\xx b)\|^{1/2}}}=\displaystyle{\frac{\|b^*a^*ab\|^{1/2}}{\|b^*b\|^{1/2}}}=\displaystyle{\frac{\|ab\|}{\|b\|}}\le \|a\|.
$$
On the other hand, taking $\yy=\xx$, one gets $\displaystyle{\frac{\|\theta(\varphi(\xx),\varphi(\xx))\|^{1/2}}{\|\theta(\xx,\xx)\|^{1/2}}}=\|a\|$.
\end{proof}
\begin{teo}
Let $q\in\q_\rho$ and $\XX\in(T\q_\rho)_q$. Then, with the above definition \ref{norma de phi}, one has
$$
|\langle \XX,\XX\rangle_q|=\|\XX\|_q.
$$
\end{teo}
\begin{proof}
Since both distributions of norms are $\u(\theta)$-invariant, it suffices to consider the case $q=p=\frac12 (\rho+I)=\left(\begin{array}{cc} 1 & 0 \\ 0 & 0 \end{array}\right)$.  For  the range of this projection we can choose the basis $\ee_1$. If $\XX\in(T\q_\rho)_{p}$ then it  is  anti-Hermitian and co-diagonal, i.e.,  $\XX=\left(\begin{array}{cc} 0 & x_{12} \\ -x_{12}^* & 0 \end{array}\right)$. Then
$$
-\theta(\kappa_{\ee_1}(\XX),\kappa_{\ee_1}(\XX))=-\left(\begin{array}{cc} 1 & 0 \end{array}\right)\left(\begin{array}{cc} 0 & x_{12}\\ -x_{12}^* & 0 \end{array}\right) \rho \left(\begin{array}{cc} 0 & x_{12}\\ -x_{12}^* & 0 \end{array}\right) \left(\begin{array}{c} 1 \\  0 \end{array}\right)=x_{12}x_{12}^*.
$$
Therefore, by the above lemma, $|\langle\XX,\XX\rangle_{p}|=\|x_{12}x_{12}^*\|^{1/2}=\|x_{12}\|$.

On the other hand, if $q=p$, then clearly $\|\ \|_{p}$ is the usual norm in $M_2(\a)$. Then
$$
\|\XX\|_{p}=\|\XX\|=\|\XX^*\XX\|^{1/2}=\|\left(\begin{array}{cc} 0 & x_{12}\\ -x_{12}^* & 0 \end{array}\right)^2\|^{1/2}=\|\left(\begin{array}{cc} x_{12}x_{12}^* & 0 \\ 0  & x_{12}^*x_{12} \end{array}\right)\|^{1/2}
$$
$$
=\max\{\|x_{12}x_{12}^*\|^{1/2}, \|x_{12}^*x_{12}\|^{1/2}\}=\|x_{12}\|.
$$
\end{proof}

\section{The scalar case}
In this section we shall consider the classical case when $\a=\mathbb{C}$. We shall check that the geometry induced by $\a_\rho$ on $\d$ is the classical hyperbolic geometry of the Poincar\'e disk. In the scalar case, or more generally, when $\a$ is commutative, the coefficient bundle $\c$ consists of $\a$ in each fiber: the Hilbertian product, the moment map,  and so forth, take values in $\a$. Indeed, if $\varphi$ is an endomorphism of $R(q)$, with bases $\xx$ and $\xx u$, then the corresponding   matrices of $\varphi$ in these bases are
$$
a \ \hbox{ and } uau^*=a.
$$
That is, the coefficient bundle $\c$ ends up being $\a$. 

In particular, in the case $\a=\mathbb{C}$, one obtains a complex structure for the unit disk $\d$. The goal of this section is to show that this structure is the classical complex structure of the Poincar\'e disk.

We first recall the isomorphism 
$$
\Phi_\d:\d\to\q_\rho.
$$
In the scalar case we get
$$
\d\ni z \mapsto p_z=\displaystyle{\frac{1}{1-|z|^2}} \left(\begin{array}{cc} 1 & -\bar{z} \\ z & -|z|^2 \end{array}\right).
$$
The tangent spaces $(T\d)_x$ identify with $\mathbb{C}$. Given $a\in\mathbb{C}$, regarded as a tangent vector in $T(\d)_z$, let us denote by $\XX_a=(d\Phi_\d)_z(a)$ the corresponding tangent vector in $T(\q_\rho)_{p_z}$. Clearly,  one gets
$$
\XX_a=\displaystyle{\frac{1}{(1-|z|^2)^2}}\left( \begin{array}{cc} \bar{a}z+a\bar{z} & -\bar{a}-a \bar{z}^2 \\ a+\bar{a}z^2 & -\bar{a}z-a\bar{z} \end{array} \right). 
$$
\subsection{The complex inner product}
For the module $R(p_z)$, we chose  the basis $\xx_z=\displaystyle{\frac{1}{(1-|z|^2)^{1/2}}}\left(\begin{array}{c} 1 \\ z \end{array} \right)$. Thus, the lifting $\kappa_{\xx_z}(\XX_a)$ is given by
$$
\XX_a\xx_z=\displaystyle{\frac{1}{(1-|z|^2)^{5/2}}}\left( \begin{array}{cc} \bar{a}z+a\bar{z} & -\bar{a}-a\bar{z}^2 \\ a+\bar{a}z^2 & -\bar{a}z-a\bar{z} \end{array} \right) \left(\begin{array}{c} 1 \\ z \end{array} \right)=\displaystyle{\frac{1}{(1-|z|^2)^{3/2}}}\left(\begin{array}{c} a\bar{z} \\ a \end{array} \right).
$$
Therefore, if $z\in\d$ and  $a,b\in\mathbb{C}$ ($=(T\d)_z$)
\begin{equation}\label{producto caso escalar}
\langle a,b\rangle_z=-\theta(\XX_a\xx_z,\XX_b\xx_z)=-\displaystyle{\frac{1}{(1-|z|^2)^{3/2}}}\theta(\left(\begin{array}{c} a\bar{z} \\ a \end{array} \right),\left(\begin{array}{c} b\bar{z} \\ b \end{array} \right))=
\displaystyle{\frac{\bar{a}b}{(1-|z|^2)^{2}}},
\end{equation}
which is the classical complex inner product in the Poincar\'e disk.

\subsection{The linear connection}
Let $\XX$ be a tangent field defined on a neighbourhood of $q$ in $\q_\rho$, and $\YY\in(T\q_\rho)_q$. Let $q(t)$ be a smooth curve adapted to $\YY$: $q(0)=q$ and $\dot{q}(0)=\YY$. Then the covariant derivative in $\q_\rho$ is given by \cite{cpr}
$$
\nabla_\YY \XX_q=\frac{d}{dt}\XX_{q(t)}|_{t=0}+[\XX_q,[\YY,q]].
$$

Let now $a=a(z)$ be a $\mathbb{C}$-valued smooth map defined on a neighbourhood of $z_0\in\d$, regarded as a tangent vector field in $\d$, and $b\in\mathbb{C}$ a tangent vector at $z_0$. To compute $\nabla_ba_{z_0}$, we have to compute
$\nabla_{\XX_b}\XX_a$ at $z=z_0$, and identify this tangent vector as a matrix $\XX_c$ for certain $c\in\mathbb{C}$ (at $z$!), and then $c=\nabla_b a_{z_0}$.
In order to simplify this conputation, we shall consider the case $z_0=0$. In fact, due to the invariance of the linear connection under the action of $\u(\theta)$, this will suffice to identify the covariant derivative.
In this case ($z_0=0$), we have
$$
\XX_a(z)=\displaystyle{\frac{1}{(1-|z|^2)^2}} \left( \begin{array}{cc} \bar{a}(z)z+a(z)\bar{z} & -\bar{a}(z)-a(z)\bar{z}^2 \\ a(z)+\bar{a}(z)z^2 & -\bar{a}(z)z-a(z)\bar{z} \end{array} \right)\ , \  \XX_b=\left(\begin{array}{cc} 0 & -\bar{b} \\ b & 0 \end{array}\right).
$$
Applying the formula above and choosing a smooth $z(t)\in\d$ such that $z(0)=0$ and $\dot{z}(0)=b$ (for instance $z(t)=tb$),  after straightforward computations one obtains
$$
\nabla_{X_a} (X_{b})_{E_1}=\left( \begin{array}{cc} 0 & -\frac{\partial}{\partial b} \bar{a}(0) \\ \frac{\partial}{\partial b} a(0) & 0\end{array}\right)=X_{\frac{\partial}{\partial b} a(0)}\ \ \hbox{ at } z=0.
$$
Thus, 
\begin{equation}\label{conexion escalar}
\nabla_b a_{0}=\frac{\partial}{\partial b} a(0).
\end{equation}
This coincides with the Levi-Civita connection of the classical metric of the Poincar\'e disk (at the origin).
\subsection{The moment map}
Note that the Lie algebra $\mathfrak{U}(\theta)$ is given in this case by all matrices of the form
$$
{\tilde{a}}=\left(\begin{array}{cc} i\alpha & \omega \\ \bar{\omega} & i\beta \end{array}\right)= i\left(\begin{array}{cc} \alpha & 0 \\ 0 & \beta \end{array}\right)+\left(\begin{array}{cc} 0 & \omega \\ \bar{\omega} & 0 \end{array}\right) \  , \  \alpha,\beta\in\mathbb{R} , \ \omega\in\mathbb{C}.
$$
Recall that if $q=p_\xx\in\q_\rho$ and ${\tilde{a}}\in \mathfrak{U}(\theta)$, then the moment map is given by
$$
f_{\tilde{a}}(q))=\frac{1}{2i}\theta(\xx,{\tilde{a}}\xx).
$$
If $z\in\d$ and  $q=p_z$,  we choose as above the basis $\displaystyle{\frac{1}{(1-|z|^2)^{1/2}}}\left(\begin{array}{c} 1 \\ z \end{array}\right)$, and ${\tilde{a}}$ is given as above, then
$$
f_{\tilde{a}}(z))=f_{\tilde{a}}(q)=\displaystyle{\frac{1}{1-|z|^2}}\left\{\frac12 \theta(\left(\begin{array}{c} 1 \\ z \end{array}\right), \left(\begin{array}{cc} \alpha & 0  \\ 0 & \beta \end{array}\right)\left(\begin{array}{c} 1 \\ z \end{array}\right))+\frac{1}{2i}\theta(\left(\begin{array}{c} 1 \\ z \end{array}\right), \left(\begin{array}{cc} 0 & \omega \\ \bar{\omega} & 0  \end{array}\right)\left(\begin{array}{c} 1 \\ z \end{array}\right)\right\} 
$$
\begin{equation}\label{momento escalar}
=\displaystyle{\frac{1}{1-|z|^2}}\left( \frac12(\alpha-\beta|z|^2)+\frac{1}{2i}(\omega z- \bar{\omega}\bar{z})\right).
\end{equation}

\subsection{Commutative C$^*$-algebras}
If $\a$ is commutative (i.e., $\a=C(\Omega,\mathbb{C})$ for some compact Hausdorf space $\Omega$), the coefficient bundle also reduces to $\a$, as remarked at the beginning of this section. Moreover, it is clear that the computations done in this section can be carried over exactly  in the same way. Thus, one gets a near classical situation, in which the Hilbertian product, the metric and the moment map take values in $\a$, and the formulas look the same as in the scalar case, replacing complex numbers by continuous functions.

\section{Valuations}

In this section we introduce valuation maps. Once a valuation map onto a commutative C$^*$-algebra   is chosen, the non commutative K\"ahler structure becomes  a classical K\"ahler structure: measurements take values in  a fixed scalar field,  instead of being elements of the coefficient bundle $\c$. Most important, valuation maps will allow us to examine the convexity properties of the moment map.

In what follows, $\f$ denotes a commutative C$^*$-algebra. 
\begin{defi}
A {\it valuation} $\nu$ in $\q_\rho$ is a  differentiable map $\nu:\c\to \f$ with the following properties:
\begin{enumerate}
\item 
$\nu$ is positive in the following sense: for any $q\in\q_\rho$, $\nu|_{\c_q}:\c_q\to \f$ is a positive linear map between C$^*$-algebras. In particular, this implies that $\nu|_{\c_q}$ is bounded.
\item
$\nu$ is tracial: for any $q\in\q_\rho$ and $a,b\in\c_q$, $\nu(ab)=\nu(ba)$.

{\rm Additionally, we say that $\nu$ is {\it faithful} if }
\item
for any $q\in\q_\rho$ and $a\in\c_q$, $\nu(a^*a)=0$ implies $a=0$.
\end{enumerate}
\end{defi}
Let us introduce the following examples, which show that the existence of $\nu$ is not an unlikely event.
\begin{ejems}

\noindent

\begin{enumerate}
\item
If the base algebra $\a$ admits a trace $\tau$ with values in a commutative subalgebra $\f\subset \a$, then naturally a valuation $\nu$ is defined: $\nu(\varphi)=\tau(a)$, where $a$ is the matrix of $\varphi$ in $\xx$, for any $q\in\c_q$.
\item
$\a$ need not admit a trace. Suppose that there exists a $*$-homomorphism onto a commutative algebra $\pi:\a\to \f$ (here $\f$ need not be a subalgebra of $\a$). In this case, $\pi(ab)=\pi(a)\pi(b)=\pi(b)\pi(a)=\pi(ba)$.  One such example is the Toeplitz C$^*$-algebra $\t(C(\mathbb{T}))=\{T_f: f\in C(\mathbb{T})\}$, where $T_f$ denotes the Toeplitz operator with symbol $f$. $\t(C(\mathbb{T}))$ does not admit a trace, but it has a $*$-homomorphism onto a commutative algebra, namely 
$$
\pi:\t(C(\mathbb{T}))\to \t(C(\mathbb{T}))/\k(L^2(\mathbb{T}))\simeq C(\mathbb{T}).
$$
Another example of this sort is the algebra $\d+\k=\{D+K: D \hbox{ diagonal  and } K \hbox{ compact}\}\subset\b(\ell^2).$
In this case, there is a homomorphism
$$
\pi: \d+\k\to \d+\k/\k(\ell^2)\simeq \ell^\infty/c_0.
$$
\item
A third sort of example is obtained if $\a$ has a positive tracial map onto a  commutative algebra.
\end{enumerate}
\end{ejems}

Let us fix  a valuation $\nu:\c\to \f$.

The structures that we have defined in $\q_\rho$, with values in $\c$, have valuations which transform them in structures with values in $\f$. The main quantity is the $\c$-valued Hilbertian product: if $q=p_\xx$,
$$
\langle\XX,\YY\rangle_q=-\theta(\kappa_\xx(\XX),\kappa_\yy(\YY))=-\theta(\XX\xx,\YY\yy)\in\c_q,
$$
and applying $\nu$
\begin{equation}\label{producto nu}
\langle \XX, \YY  \rangle_q^\nu=-\nu\theta(\XX,\xx,\YY\yy).
\end{equation}
This  $\f$-valued inner product defines a Hermitian structure in $\q_\rho$, where the scalar "field" is $\f$. Note that if $\nu$ is not faithful, $\langle\ ,\  \rangle^\nu$ is positive semi-definite.

Next, we consider the connection and curvatures, and the symplectic form. We saw in Section 6 that the canonical connection in the tautological bundle $\xi\to \q_\rho$ has curvature equal to the imaginary part of the Hilbertian product (both terms in this assertion, considered as $2$-forms in $\q_\rho$ with values in $\c$).

Let us consider now the following: pick a field of bases $\xx\in\k$ defined on an open set  and a cross section $\sigma$ of $\xi$ on this open set. We can write
$$
D_\XX\sigma=\xx(\XX\cdot a+\alpha(\XX)a)
$$
where $\sigma=\xx a$,  $a$ is an $\a$-valued function,  $\XX\cdot a$ is the directional derivative of $a$ in the direction $\XX$, and  $\alpha(\XX)$ is an $\a$-valued  $1$-form in $\q_\rho$ (what we called  the $1$-form of $\q_\rho$ in the basis $\xx$) . If we compute
$$
D_\XX D_\YY\sigma-D_\YY D_\XX\sigma-D_{[\XX,\YY]}\sigma
$$
for fields $\XX,\YY$ in $(T\q_\rho)_q$, we obtain  the following expression for the curvature $R(\XX,\YY)\sigma$:
$$
R(\XX,\YY)\sigma=\xx((d\alpha)(\XX,\YY)a+[\alpha(\XX),\alpha(\YY)] a),
$$
where $[\alpha(\XX),\alpha(\YY)]=\alpha(\XX)\alpha(\YY)- \alpha(\YY)\alpha(\XX)$.
On the other hand, we saw in Theorem \ref{61} that 
$$
-\frac{1}{2i} R(\XX,\YY)_q=Im \ \langle\XX,\YY\rangle_q,
$$
where both terms are $\c$-valued, and $\XX,\YY\in(T\q_\rho)_q$. Applying the valuation $\nu$ we get
$$
\nu(R(\XX,\YY)_q)=\nu(d\alpha(\XX,\YY)+[\alpha(\XX),\alpha(\YY)])=\nu d\alpha(\XX,\YY).
$$
Note that $\nu(\XX\cdot \alpha(\YY))=\XX\cdot d\alpha(\YY)$, because $\nu$ is linear and  bounded (it commutes with the derivatives). Hence
$$
\nu \  d\alpha(\XX,\YY)=\nu(\XX\cdot \alpha(\YY)-\YY\alpha(\XX)-\alpha([\XX,\YY]= \XX \cdot \nu\alpha(\YY) - \YY\cdot \nu\alpha(\XX)-\nu\alpha([\XX,\YY]).
$$
In other words: $\nu\  d\alpha(\XX,\YY)=d\ \nu\alpha (\XX,\YY)$. Therefore
$$
Im \  \langle(\XX,\YY\rangle_q^\nu=\nu \ Im \langle(\XX,\YY\rangle_q=-\frac{i}{2i} \nu R(\XX,\YY)_q=-\frac{1}{2i} \nu \ d\alpha(\XX,\YY)=-\frac{1}{2i} d\ \nu\alpha(\XX,\YY).
$$
Thus, recalling that $\omega(\XX,\YY)=Im \ \langle\XX,\YY \rangle$, we have that the alternate $\f$-valued $2$-form $\nu\omega$ over $\q_\rho$ satisfies the equality
\begin{equation}\label{nu omega}
\nu\omega(\XX,\YY)=-\frac{1}{2i} d\  \nu\alpha (\XX,\YY).
\end{equation}
We shall call $\nu\omega $ the {\it valuated symplectic form} over $\q_\rho$, with values in $\f$. Then we have that the valuated  symplectic  form $\nu\omega$ is {\it exact}. Indeed, there exist global bases $\xx$ defined in the whole $\q_\rho$.

Recall the moment map defined in Section 7: for $\tilde{a}\in\mathfrak{U}(\theta)$, $f_{\tilde{a}}:\q_\rho\to \c$, is defined by $f_{\tilde{a}}(q)=\frac{1}{2i}\Big[(\xx,\xx^*\rho\tilde{a}\xx)\Big]$. Recall also the equality
$D_\XX f_{\tilde{a}}=\omega(\XX,\XX_{\tilde{a}})$. Let us apply the valuation to the moment map
$$
f_{\tilde{a}}^\nu=\nu f_{\tilde{a}}:\q_\rho\to \f.
$$
Note that 
$$
\XX\cdot f_{\tilde{a}}^\nu=X\cdot \nu f_{\tilde{a}}=\nu D_\XX\cdot f_{\tilde{a}},
$$
Therefore $\XX\cdot f_{\tilde{a}}^\nu=\nu \omega(\XX,\XX_{\tilde{a}})$. This means that $f^\nu$ is a  moment map. Let us understand now the map $f^\nu$ as a map from $\q_\rho$ to the dual of $\mathfrak{U}(\theta)$. Here "dual" means the space of bounded  linear "functionals" with values in $\f$. To do this, we shall consider the following tracial functional :
$$
\tau:M_2(\a)\to \f \ , \ \ \tau\left( \begin{array}{cc} a_{11} & a_{12} \\  a_{21} & a_{22} \end{array} \right) =\frac12\nu(a_{11}+a_{22}).
$$
Apparently, $\tau$ is positive, $\tau(I)=1$  and verifies $\tau(\tilde{a}\tilde{b})=\tau(\tilde{b}\tilde{a})$.
Then, if $q=p_\xx$,  we have
$$
f_{\tilde{a}}^\nu(q)=\nu \theta(\xx,\tilde{a}\xx)=\nu(\xx^*\rho\tilde{a}\xx)=\tau(\xx^*\rho\tilde{a}\xx),
$$
where the last equality uses the fact that, on elements $a$ of $\a$ (regarded as scalar matrices $\left(\begin{array}{cc} a & 0 \\ 0 & a \end{array} \right)$), $\nu$ coincides with $\tau$. Then, using the trace property of $\tau$, $\tau(\xx^*\rho\tilde{a}\xx)=\tau(\xx\xx^*\rho\tilde{a})=\tau(q\tilde{a})$, i.e.,
\begin{equation}\label{tau momento}
f_{\tilde{a}}^\nu(q)=\tau(q\tilde{a}).
\end{equation}
This formula allows one to clearly identify (by means of $\nu$) the moment map as a map from $\q_\rho$ with values in the tangent space of the Lie algebra $\mathfrak{U}(\theta)$ of the group $\u(\theta)$: to $q\in\q_\rho$ corresponds the $\f$-valued  linear functional $\tau(q \ \cdot)$, with density matrix $q$.
\begin{rem}
The dual space considered is subordinated to the valuation $ \nu$. It consists of  bounded linear functionals defined on $\mathfrak{U}(\theta)$, with values in $\f$. If, additionally, the algebra of scalars $\f$  in which one chooses to take measurements is a subalgebra of $\a$, then these functionals $\tau(q\ \cdot)$ are also $\f$-linear.
\end{rem}

Next we shall discuss a property of the moment map, which mimicks the theorem of \cite{kostant}, \cite{atiyah} and \cite{guillemin} on compact symplectic manifolds acted by a torus.  We shall consider therefore a subgroup of the full group $\u(\theta)$ acting in $\q_\rho$, namely the diagonal group
$$
\d(\theta)=\{\left(\begin{array}{cc} u_1 & 0 \\ 0 & u_2 \end{array}\right) : u_i\in\u_\a\}\subset \u(\theta),
$$
which will play the role of a torus.
Note that $\d(\theta)=\u(\theta)\cap\u_2(\a)$, i.e., the unitary matrices in $M_2(\a)$ which preserve the form $\theta$.

In classical symplectic geometry, the restriction of the action to the subgroup induces a restricted moment map, which, when regarded as a map from the manifold to the dual of the Lie algebra of the acting group, consists in composing the moment of the full group with the projection of the Lie algebra of the full group onto the Lie algebra of the subgroup.

In our case, we shall restrict the action to $\d(\theta)$. Note that the Lie algebra of $\d(\theta)$ is the subalgebra $\mathfrak{U}_0(\theta)$ given in (\ref{mathfrak}) 
$$
\mathfrak{U}_0(\theta)=\{\left( \begin{array}{cc} a_1 & 0 \\ 0 & a_2 \end{array}\right): a_i^*=-a_i\}.
$$
The projection is given by
$$
\mathfrak{U}(\theta)\to\mathfrak{U}_0(\theta) \ , \ \ \left( \begin{array}{cc} a_1 & b \\ b^* & a_2 \end{array}\right)\mapsto \left( \begin{array}{cc} a_1 & 0 \\ 0 & a_2 \end{array}\right).
$$ 
The following  is an Atiyah-Guillemin-Sternberg's type of convexity result  \cite{kostant},\cite{atiyah}, \cite{guillemin}.
\begin{teo}
The image of   moment map $f^\nu$ of the restricted action of the group $\d(\theta)$, 
$$
\q_\rho\ni q\stackrel{f^\nu}{\longmapsto} \tau(q\ \cdot )\in\ele(\mathfrak{U}_0(\theta),\f),
$$
regarded as a subset of the  space $\ele(\mathfrak{U}_0(\theta),\f)$ of linear functionals $\mathfrak{U}_0(\theta)\to \f$, is a convex set.
\end{teo}
\begin{proof}
As done previously in Section 9, we shall use the model $\d\simeq\q_\rho$, to prove our statement. Recall that each $q=p_z$, for some $z\in\d$, where $p_z=(1-zz^*)^{-1}\left(\begin{array}{cc} 1 & -z^* \\ z & -zz^*\end{array} \right)$. Then, if $\tilde{a}_0=\left(\begin{array}{cc} a_1 & 0 \\ 0 & a_2 \end{array} \right)\in\mathfrak{U}_0(\theta)$, we get
$$
f_{\tilde{a}}(p_z)=\tau \Big( (1-zz^*)^{-1}\left(\begin{array}{cc} 1  &  -z^* \\ z & -zz^* \end{array}\right) \left( \begin{array}{cc} a_1 & 0 \\ 0 & a_2 \end{array}\right)\Big)=\nu((1-zz^*)^{-1}(a_1-zz^*a_2))
$$
$$
=\nu\Big(\left( \begin{array}{cc} (1-zz^*)^{-1} & -(1-zz^*)^{-1}zz^* \end{array}\right) \left(\begin{array}{c} a_1 \\ a_2 \end{array} \right) \Big).
$$
Therefore, if the elements of $\mathfrak{U}_0(\theta)$ are represented as pairs $(a_1,a_2)\in\a^2_{ah}$, via the valuation $\nu$,  the image of the moment map of the group $\d(\theta)$ identifies with the set of pairs 
$$
\{(c_1,c_2)\in\a^2: c_1\in G^+, c_1+c_2=1\},
$$  
which is clearly a convex set.
\end{proof}

\section{The Poincar\'e disk as a "cotangent bundle"}

In this appendix we present the Poincar\'e disk as the {\it dual} $T^*G^+$ of the tangent bundle $TG^+$, represented as the Poincar\'e half-space $\h$ of $\a$. 

We give a brief the description of the half-space model $\h$ of $\q_\rho$. Recall that 
$$
\h=\{\zeta =x+i y\in\a: x^*=x , y\in G^+\}.
$$
The quadratic form $\theta_H$ of $\h$ is given by 
$$
\theta_H(\xx,\yy)=\xx^* \rho_H \yy=-i (x_1^*y_2-x_2^*y_1),
$$
where $\rho_H=\left( \begin{array}{cc} 0 & -i \\ i & 0 \end{array} \right)$. The hyperboloid $\k_H$ is then
$$
\k_H=\{\xx\in\a^2: x_1\in G \hbox{ and } \theta_H(\xx,\xx)=2\  Im\ x_1^*x_2=1\}.
$$
Note the fact that $x_1\in G$ implies that also $x_2\in\ G$. As in the other model, every $\xx\in\k_H$ gives rise to a projection (its representative in $\q_\rho$):
$p_\xx=\xx\xx^*\rho_H$.  Consider its complement
$$
1-p_\xx=\left(\begin{array}{cc} 1-i x_1x_2^* & i x_1x_1^* \\ -i x_2x_1^* & 1 + i x_2 x_1^* \end{array} \right).
$$
Note the second column $\left( \begin{array}{c} ix_1 x_1^* \\ 1+i x_2x_1^*  \end{array} \right)=\left( \begin{array}{c} ix_1  \\ (x_1^*)^{-1}+i x_2  \end{array} \right) x_1^*$, which means that $\left( \begin{array}{c} ix_1  \\ (x_1^*)^{-1}+i x_2  \end{array} \right)$ generates the nullspace of $p_\xx$. Let us denote this vector by
\begin{equation}\label{x raro}
\xx_{\perp}:=\left( \begin{array}{c} ix_1  \\ (x_1^*)^{-1}+i x_2  \end{array} \right).
\end{equation}
Note that $\theta_H(\xx_\perp, \xx_\perp)=-1$. . The pair $\{\xx,\xx_\perp\}$ forms a $\theta_H$-orthogonal basis for $\a^2$. Any element $\zz\in A^2$ is written
$$
\zz=\xx\theta_H(\xx,\zz)- \xx_\perp \theta_H(\xx_\perp,\zz).
$$
Note also that $\xx_\perp=\xx i + \ee_2 (x_1^*)^{-1}$, which gives $\xx_\perp=\xx-\ee_2 i (x_1^*)^{-1}$.

The map  $\k_H\to\h$, $\xx\mapsto \zeta=x_2x_1^{-1}$  is the $\h$ valued version of the projection map $\xx\to p_\xx$. Its differential at $\xx$, sends the $\vv\in(T\k_H)_\xx$ to
$$
\vv\mapsto \dot{v}:=v_2x_1^{-1}- x_2 x_1^{-1}v_1 x_1^{-1}.
$$
Suppose now that $\vv\in N(p_\xx)$: $\vv=\left(\begin{array}{c} x_1 \\ x_2-i(x_1^*)^{-1}\end{array}\right) \lambda$, so that 
$$
\dot{v}=(x_2-i(x_1^*)^{-1})\lambda x_1^{-1}- x_2x_1^{-1}x_1\lambda x_1^{-1}=(x_2\lambda-i(x_1^*)^{-1}\lambda-x_2\lambda)x_1^{-1}=-i)(x_1^*)^{-1}\lambda x_1^{-1},
$$
and thus $\lambda=i\ x_1^* \dot{v} x_1$. Then, given $\xx$ in $\k_H$  which projects over $\zeta$ in $\h$, and $\dot{v}\in (T\h)_\zeta$, we have the  (lifting) form, as in Section 2
$$
\kappa^H_\xx(\dot{v})=\left( \begin{array}{c} x_1 \\ x_2-i \ x_1(x_1^*)^{-1}\end{array} \right) i x_1^*\dot{v} x_1=\left(\begin{array}{c} x_1x_1^* \\ x_2x_1^*-i\end{array} \right) i \ \dot{v} x_1.
$$
Or equivalently 
$$
\kappa^H_\xx(\dot{v})=(\xx \ i -\ee_2 i \ (x_1^*)^{-1}) i \ x_1^*\dot{v} x_1=(-\xx x_1^* +\ee_2) \dot{v} x_1.
$$
There is a natural global cross section for the map $\k_H\ni\xx\mapsto \zeta\in\h$, namely
$$
\zeta=x+i y \mapsto \left(\begin{array}{c} 1 \\ \zeta \end{array}\right) (2y)^{-1/2}.
$$
Then we have 
\begin{equation}\label{asterisco}
\kappa^H_\xx(\dot{v})=\left(\begin{array}{c} (2y)^{-1} \\ \zeta (2y)^{-1}-i\end{array} \right) i \ \dot{v} (2y)^{-1/2}.
\end{equation} 

We want to compute the Hilbertian product in $\h$. Consider the expressions (\ref{asterisco}) for two tangent  vectors $\dot{v}, \dot{w}$ in $(T\h)_\zeta$:
$$
\vv=\left(\begin{array}{c} (2y)^{-1} \\ \zeta (2y)^{-1}-i\end{array} \right) i \ \dot{v} (2y)^{-1/2} \ , \ \
\ww=\left(\begin{array}{c} (2y)^{-1} \\ \zeta (2y)^{-1}-i\end{array} \right) i \ \dot{w} (2y)^{-1/2}.
$$
Then we have
$$
\theta_H(\vv,\ww)=\langle\dot{v},\dot{w}\rangle= (2y)^{-1/2} \dot{v}^* \theta_H\Big( \left(\begin{array}{c} (2y)^{-1}\\ \zeta (2y)^{-1} -1\end{array}\right), \left(\begin{array}{c} (2y)^{-1}\\ \zeta (2y)^{-1} -1\end{array}\right)
\Big) \dot{w} (2y)^{-1/2}
$$
$$
= -(2y)^{-1/2}\dot{v}^* (2y)^{-1} \dot{w} (2y)^{-1/2}=-\frac14 (y^{-1/2} \dot{v}y^{-1/2})^*(y^{-1/2} \dot{w}y^{-1/2}).
$$
Suppose now that the algebra  has a trace $\tau$ onto a commutative subalgebra $\b\subset\a$. If we take the trace of the above expression, we get
\begin{equation}\label{traza del producto}
\tau\langle \dot{v},\dot{w}\rangle=-\tau((2y)^{-1}\dot{v}^*(2y)^{-1} \dot{w}).
\end{equation}

Here $y^{-1/2} \dot{v}y^{-1/2}$ is the translation to the point $y^{-1/2} \zeta y^{-1/2}$ of the vector $\dot{v}$.  Note that $y^{-1/2} \zeta y^{-1/2}=y^{-1/2} x y^{-1/2}+i$.

Therefore $\langle \dot{v},\dot{w}\rangle_\zeta= \langle \dot{v}_0,\dot{w}_0\rangle_{\zeta_0}$, where $\zeta_0=y^{-1/2} x y^{-1/2}+i$.
When the imaginary part equals $1$, the inner product has the simpler expression $\langle \dot{v}_0,\dot{w}_0\rangle_{\zeta_0}=-\frac14 \dot{v}_0^* \dot{w}_0$. Or explicitely, if $\dot{v}_0=\dot{x}+i\ \dot{y}$ and $\dot{w}_0=\dot{\xi}+i\ \dot{\eta}$,
then $\dot{v}^*_0\dot{w}_0=(\dot{x}\dot{\xi}+\dot{y}\dot{\eta})+i (\dot{x}\dot{\eta}-\dot{y}\dot{\xi})$, and
$$
\tau\langle\dot{v}_0,\dot{w}_0\rangle_{\zeta_0}=-\frac14\{\tau(\dot{x}\dot{\xi}+\dot{y}\dot{\eta})+i\ \tau(\dot{x}\dot{\eta}-\dot{y}\dot{\xi})\}.
$$
Both traces in the above expression are selfadjoint elements of $\b$ (or real numbers if the trace is numerical). The imaginary part of $\tau\langle\dot{v}_0,\dot{w}_0\rangle_{\zeta_0}$ is  essentially the  (trace of the) symplectic form $\omega$. This expression corresponds strictly to $dp\wedge dq$ in the classical setting, where $p$ varies in the imaginary  part and $q$ in the real part. This shows that $\h=TG^+$ as a true dynamical system.

Let us write down the real and imaginary part of $\tau\langle \ , \ \rangle$ at any point $\zeta\in\h$ (non necessarily with $Im \ \zeta=1$). With the current notation, we get
\begin{equation}\label{traza del producto 2}
\tau\langle\dot{v},\dot{w}\rangle_\zeta=-\frac14\{\tau(\dot{x}\dot{\xi}+\dot{y}\dot{\eta})+i\ \tau(\dot{x}\dot{\eta}-\dot{y}\dot{\xi})\},
\end{equation}
where $\dot{x}+i\ \dot{y}=y^{-1/2}\dot{v}y^{-1/2}$ and $\dot{\xi}+i\ \dot{\eta}=y^{-1/2}\dot{w}y^{-1/2}$.

Let us compute now the {\it $1$-form of Liouville} which induces the symplectic form $\omega$ of $\h$. The $1$-form of Liouville is a form on  the cotangent bundle of the manifold. Here it  shall be presented as a $1$-form in the {\it tangent} bundle $TG^+\simeq\h$, identifying the tangent bundle $TG^+$ with the co-tangent bundle $T^*G^+$. To perform this identification, we shall use the  trace $\tau$, and the action of $G$ on $G^+$. Given a vector $\dot{v}\in(T\h)_\zeta$, for $\zeta=x+i\ y\in\h$, the Liouville form $\alpha$ maps $\zeta$ into an element of $\b$. We must project $\dot{v}$ from its tangency point $\zeta$ to the point $y\in G^+$, having in mind that $\zeta=x+i \ y$ represents the tangent $x$ at the point $y$. Finally, we evaluate $x$ in $\dot{v}$. To perform this task, we translate to $y=1$, and compute
\begin{equation}\label{1 forma de liouville}
\alpha_y(\dot{v})=\tau(y^{-1/2} x y^{-1/2} y^{-1/2} \dot{v} y^{-1/2})=\tau(y^{-1} x y^{-1} \dot{v}).
\end{equation}
This is the Liouville 1-form.

Now we differentiate this $1$-form in $\h$, i.e., we compute $d\alpha(\dot{v},v')=\dot{v} \alpha(v')- v' \alpha(\dot{v})-\alpha([\dot{v}, v'])$.  Consider a function in the variables $s,t$ whose derivatives produce the fields $\dot{v}$, $v'$ when differentiated with respect to $t$  and  $s$, respectively. We must compute 
$$
\frac{\partial}{\partial t}\alpha_\zeta(v')-\frac{\partial}{\partial s}\alpha_\zeta(\dot{v}), 
$$
because there is no need to substract $\alpha([\dot{v},v'])$ since it is trivial.
Then
$$
\frac{\partial}{\partial t}\alpha_\zeta(v')=\tau(-y^{-1}\dot{v}y^{-1}x y^{-1} v'+y^{-1}\dot{x}y^{-1} y'-y^{-1}x y^{-1} \dot{y} y^{-1} y'+y^{-1} x y^{-1} \dot{y}')
$$
and
$$
\frac{\partial}{\partial s}\alpha_\zeta(\dot{v})=\tau(-y^{-1} y' y^{-1} x y^{-1} \dot{y}+y^{-1} x' y^{-1}\dot{y}-y^{-1} x y^{-1} y' y^{-1} \dot{y}+y^{-1} x y^{-1} y').
$$
Therefore
$$
(d\alpha)_\zeta(\dot{v},v')=\frac{\partial}{\partial t}\alpha_\zeta(v')-\frac{\partial}{\partial s}\alpha_\zeta(\dot{v})=\tau(y^{-1} \dot{x} y^{-1} y'-y^{-1} x' y^{-1} \dot{y}).
$$
Now, comparing this last expression with the imaginary part of the trace of the Hilbertian product form, and adapting the notation of both computations, we get 
\begin{teo}
$$
(d\alpha)_\zeta(\dot{v},\dot{w})=\omega(\dot{v},\dot{w}).
$$
\end{teo}

\begin{rem}
If $x_1,x_2\in (TG^+)_y$, we have that $\zeta_1=x_1+  i \ y$ and  $\zeta_2=x_2+  i \ y$ belong to $\h$. Define
$$
[x_1,x_2]_y=\tau((y^{-1/2}x_1y^{-1/2})(y^{-1/2}x_2y^{-1/2}))=\tau(y^{-1}x_1y^{-1}x_2).
$$
It is  a $\b$-valued, bilinear (it takes selfadjoint values) and positive semidefinite form. Note also that this form is invariant  under the action of $G$ on $G^+$. In fact,  $g\cdot y=(g^{-1})^*yg^{-1}$, for $g\in G$ and $y\in G^+$. The action defines a linear isomorphism in $\a$, so that the same formula gives the induced action in the tangent spaces of $G^+$: $g\cdot x=(g^{-1})^*xg$, if $x\in(TG^+)_y$.  A straightforward computation shows that
$$
[g\cdot x_1,g\cdot x_2]_{g\cdot y}=[x_1,x_2]_y.
$$
If $\a=M_n(\mathbb{C})$, $\b=\mathbb{C}$ and $\tau$ is the usual trace for $n\times n$ matrices, this inner product is the usual Riemannian metric for the homogeneous space of positive definite $n\times n$ matrices.

On the  other hand, if we embed $G^+\hookrightarrow \h$, by means of $y\mapsto i \ y$, i.e.,  we regard $G^+$  as the imaginary positive axis of $\h$, then for $x_1,x_2 \in(TG^+)_y$ it holds 
$$
\langle i\ x_1, i \ x_2\rangle_{i\ y}=-\frac14(y^{-1/2}i \ x_1 y^{-1/2})^*(y^{-1/2} i \ x_2 y^{-1/2})=-\frac14(y^{-1/2}x_1y^{-1/2})(y^{-1/2} x_2 y^{-1/2}),
$$
and, thus,
$$
\tau \langle i\ x_1, i \ x_2\rangle_{i\ y}=-\frac14[x_1,x_2]_y.
$$
\end{rem}
This means that the trace of the Hilbertian product, restricted to the positive imaginary axis in $\h$, gives, essentially, the Riemannian metric of the space $G^+$.

\end{document}